\documentclass[11pt]{amsart}
\usepackage{amsmath}
\usepackage{amsmath,amssymb,enumitem,ulem,multicol}
\usepackage{mathabx}
\usepackage{amssymb,tikz}
\usepackage{amsthm}
\usepackage{upgreek} 
\usepackage{bigints}
\usepackage{parskip}
\usepackage{setspace}
\usepackage{mathrsfs}
\usepackage{mathtools}
\usepackage{xcolor}
\usepackage{xcolor}
\usepackage{float}
\usepackage[utf8]{inputenc}
\DeclareUnicodeCharacter{2212}{-}
\usepackage[english]{babel}
\usepackage{graphicx}
\usepackage{subcaption}
\usepackage{tikz}
\usetikzlibrary{decorations.pathreplacing,calligraphy}
\theoremstyle{remark}

\usepackage{array}
\newcolumntype{P}[1]{>{\centering\arraybackslash}p{#1}}

\def\R{\mathbb{R}}

\def\cN{\mathcal{N}}
\def\cA{\mathcal{A}}
\def\cK{\mathcal{K}}
\def\cT{\mathcal{T}}
\def\cE{\mathcal{E}}

\def\p{\partial}

\def\[{\partial}
\def\O{\Omega}
\def\ssT{{\scriptscriptstyle T}}
\def\HT{{H^2(\O,\cT_h)}}
\def\mean#1{\left\{\hskip -5pt\left\{#1\right\}\hskip -5pt\right\}}
\def\jump#1{\left[\hskip -3.5pt\left[#1\right]\hskip -3.5pt\right]}
\def\smean#1{\{\hskip -3pt\{#1\}\hskip -3pt\}}
\def\sjump#1{[\hskip -1.5pt[#1]\hskip -1.5pt]}
\def\jumptwo{\jump{\frac{\p^2 u_h}{\p n^2}}}
\def\b#1{\boldsymbol{#1}}
\def\norm #1{{\left\vert\kern-0.25ex\left\vert\kern-0.25ex\left\vert #1 
    \right\vert\kern-0.25ex\right\vert\kern-0.25ex\right\vert}}
\usepackage{ragged2e}
\usepackage{amssymb}
\usepackage{amsmath}
\usepackage{amsthm}
\usepackage{amsbsy}
\usepackage{psfrag}
\usepackage{pstricks}
\usepackage{graphics}
\usepackage{graphicx}
\usepackage{tgtermes}
\fontfamily{qcr}
\theoremstyle{plain}
\newtheorem{theorem}{Theorem}[section]
\newtheorem{lemma}[theorem]{Lemma}

\newtheorem{example}[theorem]{Example}

\theoremstyle{remark}
\newtheorem{remark}[theorem]{Remark}

\begin{document}
\allowdisplaybreaks[4]
\numberwithin{figure}{section}
\numberwithin{table}{section}
 \numberwithin{equation}{section}
% \numberwithin{figure}{section}
%
\title[ Adaptive modified weak galerkin method for Obstacle Problem]
 {Adaptive modified weak galerkin method for Obstacle Problem}
\author{Tanvi Wadhawan}\thanks{}
\address{Department of Mathematics, Harish Chandra Research Institute, Prayagraj- 211019}
\email{tanviwadhawan1234@gmail.com}
\date{}

\begin{abstract}
This article introduces a novel residual-based \textit{a posteriori} error estimators for the Modified Weak Galerkin (MWG) finite element method applied to the obstacle problem. To the best of the author's knowledge, this work represents the first integration of the MWG method into an adaptive finite element framework for variational inequalities. The proposed error estimators is rigorously proven to be both reliable and efficient in quantifying the approximation error, measured in a natural energy norm. A key aspect of the analysis involves decomposing the discrete solution into conforming and non-conforming components, which plays a central role in the error estimation process. Numerical experiments are conducted to validate the theoretical findings, demonstrating the reliability and efficiency of the proposed error estimator.
\end{abstract}
%
%\subjclass{65N30, 65N15}
%
\maketitle
%%%%%%%%%%%%%%%%%%%%%%%%%%%%%%%%%%%%%%%%%%%%%%%%%%%%%%%%%%%%%%%%
\allowdisplaybreaks
\def\R{\mathbb{R}}
\def\cA{\mathcal{A}}
\def\cK{\mathcal{K}}
\def\cN{\mathcal{N}}
\def\p{\partial}
\def\O{\Omega}
\def\bbP{\mathbb{P}}
\def\cV{\mathcal{V}}
\def\cM{\mathcal{M}}
\def\cT{\mathcal{T}}
\def\cE{\mathcal{E}}
\def\bF{\mathbb{F}}
\def\bC{\mathbb{C}}
\def\bN{\mathbb{N}}
\def\ssT{{\scriptscriptstyle T}}
\def\HT{{H^2(\O,\cT_h)}}
\def\mean#1{\left\{\hskip -5pt\left\{#1\right\}\hskip -5pt\right\}}
\def\jump#1{\left[\hskip -3.5pt\left[#1\right]\hskip -3.5pt\right]}
\def\smean#1{\{\hskip -3pt\{#1\}\hskip -3pt\}}
\def\sjump#1{[\hskip -1.5pt[#1]\hskip -1.5pt]}
\def\jumptwo{\jump{\frac{\p^2 u_h}{\p n^2}}}

%Compared to the linear finite elements, we may compute more accurate discrete solutions to unilateral contact variational inequalities using higher order finite elements.We derive the error estimator which consists of a discrete Lagrange multiplier related to the obstacle constraint

%%%%%%%%%%%%%%%%%%%%%%%%%%%%%%%%%%%%%%%%%%%%%%%%%%%%%%%

%{New changes (Reviewer 1) are done with red color.}\\
%{New changes (Reviewer 2) are done with blue color.}

%\textcolor{violet}{New changes are done with violet color.}
\section{\textbf{Introduction}}

The obstacle problem is a fundamental mathematical model with the broad applications in mathematics, engineering, physics, thermostatics etc. It is formulated as a contact problem that seeks to determine the equilibrium configuration of an elastic membrane constrained to lie above a prescribed barrier while its boundary remains fixed. Since the contact region between the membrane and the barrier is not known \textit{a priori}, the obstacle problem belongs to the class of free boundary problems. In general, the domain of the obstacle problem can be divided into three distinct regions: the contact region, the non-contact region, and the free boundary, each of which influences the regularity of the solution. The obstacle problem serves as a prototype for variational inequalities of the first kind, where the problem is posed over a non-empty, closed, convex set.

In the 1960s, G. Stampacchia made pioneering contributions to the theory of elliptic obstacle problem \cite{1963:Stampacchia}, providing foundational results on the existence of solution, its properties, and the characterization of free boundary regions. Building on this groundwork, we consider a model obstacle problem defined on a bounded polygonal domain \( \Omega \subset \mathbb{R}^2 \) with boundary \( \Gamma \). The weak formulation of this problem seeks a function \( u \in \mathcal{K} \) such that
\begin{align}\label{equation1_obs}
\int\limits_{\Omega} \nabla u \cdot \nabla (v - u) \, dx &\geq \int\limits_{\Omega} f (v - u) \, dx, \quad \forall~ v \in \mathcal{K},
\end{align}
where the non-empty, closed, convex set \( \mathcal{K} \) is given by
\begin{align*}
\mathcal{K} &= \{ v \in V:=H^1_0(\Omega) : v \geq \psi \ \text{a.e. in} \ \Omega \},
\end{align*}
with \( f \in L^2(\Omega) \) representing the force term and \( \psi \in H^2(\Omega) \) serving as the obstacle function, satisfying \( \psi \leq 0 \) on \( \Gamma \). 

This contact problem is a versatile model with applications in membrane deformation in elasticity, non-parametric minimal surface problems, and capillary surface modeling. It is also widely used to describe phenomena like elasto-plastic torsion and cavitation in lubrication theory, which are categorized as obstacle-type problems.

To better understand the physical context, consider a homogeneous elastic membrane occupying the domain $\Omega$, subjected to a vertical load represented by $f$. The deflection of the membrane is constrained by a rigid obstacle $\psi$. The equilibrium configuration of the elastic membrane is achieved by finding a function \( u \in \mathcal{K} \) that minimizes the total potential energy, given by the functional:
\begin{align}\label{minimiz}
J(v) &= \frac{1}{2} a(v, v) - L(v), \quad v \in \mathcal{K},
\end{align}
where 
\begin{align*}
a(w, v) &:= \int\limits_\Omega \nabla w \cdot \nabla v~dx, \quad L(v) := \int_{\Omega} f v \, dx.
\end{align*}
From the calculus of variations \eqref{minimiz} is equivalent to the variational inequality \eqref{equation1_obs}. Using the Poincar$\acute{e}$ inequality \cite{BScott:2008:FEM} and the Cauchy Schwarz inequality, the following properties can be established for any \( v \in V \):
\begin{align}\label{eq:bounded1}
a(v, v) \geq \alpha \|v\|^2_{H^1(\Omega)},~~|L(v)| \leq \|f\|_{L^2(\Omega)} \|v\|_{L^2(\Omega)},
\end{align}
where \( \alpha > 0 \) is ellipticity constant. Additionally, for any \( w, v \in V \), we have:
\begin{align}
|a(w, v)| &\leq |w|_{H^1(\Omega)} |v|_{H^1(\Omega)}. \label{eq:bounded2}
\end{align}
Owing to the boundedness \eqref{eq:bounded2} and ellipticity \eqref{eq:bounded1} of the bilinear form $a(u,v)$, along with the continuity of the linear functional $L(v)$, the existence and uniqueness of the solution to the variational inequality \eqref{equation1_obs} are established by the classical Lions and Stampacchia theorem \cite{AH:2009:VI, Ciarlet:1978:FEM}.
\par 
\noindent 
Finite element methods have been extensively developed and analyzed to approximate solutions to variational problems. In this context, the regularity of the solution is crucial for understanding the qualitative behaviour of the approximation. However, irregularities can arise due to factors such as free boundaries or reentrant corners within the domain.  Due to the inherent non-smooth and non-linear nature of variational inequalities, there has been a growing interest in \textit{a posteriori} error control for
these variational inequalities and discretization methods in recent years. A key aspect of adaptive finite element methods is the development of \textit{a posteriori} error estimators, which link the error to computable variables derived from the discrete solution and available data. To reduce computational costs while maintaining accuracy, adaptive mesh generation has proven to be a highly effective strategy. The \textit{a priori} error analysis of the obstacle problem has been extensively studied in several notable works \cite{Falk:1974:VI, Glowinski:2008:VI, AH:2009:VI, quad:2002}. Similarly, significant contributions to the \textit{a posteriori} error analysis can be found in \cite{BS:2015:hpAdapt, Kporwal:2014:quad, TG:2014:VIDG1, TG:2014:VIDG, Veeser:2001:VI, WHE:2015:ApostDG} and the references therein.
\par
\noindent
Unlike the standard finite element method, the Weak Galerkin (WG) method, introduced by Wang and Ye \cite{Wang_2013} utilizes weak functions as the approximate function space and replaces the classical gradient operator with a weak gradient operator. A weak function $v=\{v_0, v_b\}$ is defined such that $v_0$ represents the value of $v$ in the interior of elements while $v_b$ represents its value on the boundary of elements. This distinct approach enables the WG method to define approximate functions separately within the interior and on the boundary of elements. One of the key features of the WG method is its ability to incorporate entirely discontinuous finite element functions, where the value of a weak function on an element's boundary does not necessarily coincide with its trace. This flexibility provides a distinct advantage over the closely related Discontinuous Galerkin (DG) finite element methods. Over time, WG methods have been further developed for solving a variety of equations, including the Helmholtz equation \cite{Mu:2015:Helmholtz, Du:2017:hemloboltz}, biharmonic equations \cite{Mu:2014:biharmonic, Wang:2013:biharmonic, Mu:2014:biharmonic1}, Stokes equations \cite{Liu:2018:Stokes, Hu:2019:Stokes, Wang:2016:Stokes}, and Maxwell equations \cite{Mu:2015:maxwell}.

To address the challenge of excessive degrees of freedom in the primal formulation of the original WG method, Wang, Malluwawadu, Gao, and McMillan in \cite{Wang_2014} introduced a modified approach called the Modified Weak Galerkin method. This modification incorporates a newly defined weak derivative by replacing $v_b$ by average of $v_0$, thereby reducing the number of unknowns compared to the WG method while preserving the same level of accuracy. The MWG method also shares the same degrees of freedom as the interior penalty discontinuous Galerkin  method. However, a notable advantage of the MWG method over the interior penalty DG method is that it does not require large stabilization parameters to ensure stability. The MWG method has also been successfully applied to solve several problems, such as Signorini and obstacle problems \cite{Zeng:2017:modobstacle}, Stokes problems \cite{Tian:2018}, reaction-diffusion equations \cite{Li:2022, Gao:2020}, and biharmonic equations \cite{Cui: 2021: biharmonicmod}.
\par 
\noindent
The primary research on adaptive MWG methods has been focused on elliptic variational equalities. For example, Zhang and Lin in \cite{Zhang:2017} proposed a residual-type error estimator and demonstrated global upper and lower bounds for the MWG method applied to second-order elliptic problems. Additionally, Xie, Cao, Chen, and Zhong \cite{Yingying:2023} introduced a new MWG finite element method for second-order elliptic problems, simplifying the bilinear form without the need for a skeletal variable. They established both convergence and quasi-optimality for the adaptive MWG method.  Furthermore, Xie et al. designed a residual-type error estimator for the MWG method applied to 2D $\b{H}$(curl) elliptic problems in \cite{Xie: 2023:modified}.This work fills a gap in the literature by pioneering the application of the adaptive Modified Weak Galerkin (MWG) method to variational inequalities. The primary challenge addressed here stems from the inherent non-linearity of obstacle problems and the non-consistency of the MWG method. Building on a novel approach, an \textit{a posteriori} error analysis is presented for the MWG finite element method, incorporating the contact condition through integral constraints. This framework aims to advance the understanding and numerical treatment of such problems. A key aspect of the approach involves decomposing the discrete solution into conforming and non-conforming components, followed by the construction of the discrete Lagrange multiplier $\lambda_h(u_h)$. This multiplier is defined as a constant polynomial over each element of the domain, utilizing the mapping $\pi^0_h$. This construction is crucial for ensuring the sign properties required to establish reliability estimates.
\par 
\noindent
The structure of the remainder of the paper is as follows: In Section 2, we introduce the necessary notations and preliminary results, followed by a detailed formulation of the MWG method for problem \eqref{equation1_obs}. Section 3 focuses on defining an auxiliary functional $\lambda(u)$ in the space $H^{-1}(\Omega)$, associated with the residual of the exact solution $u$ relative to the continuous variational formulation \eqref{equation1_obs}. Additionally, we construct the discrete Lagrange multiplier $\lambda_h(u_h)$  as a counterpart to $\lambda(u)$ in a suitable space and establish its crucial sign properties, which underpin the subsequent analysis. Section 4 is dedicated to the \textit{a posteriori} error analysis, where we define the residual functional and present the main theoretical results demonstrating the reliability of the error estimator and proving its efficiency. Finally, in Section 5, we provide numerical experiments in two dimensions to validate the theoretical findings and illustrate the performance of the \textit{a posteriori} error estimators.

\section{\textbf{The MWG Method for Obstacle Problem}}\label{sec2}
\subsection{Preliminaries}
To describe the MWG method, we first revisit the concept of weak functions and their weak derivatives as outlined in the article \cite{Wang_2013}.
\begin{itemize}
\item \textbf{Weak Function:} A weak function on domain \( D \) is defined as a pair \( v = \{v_0, v_b\} \), where \( v_0 \in L^2(D) \) represents the value of \( v \) in the interior of \( D \), and \( v_b \in L^2(\partial D) \) denotes its value on the boundary \( \partial D \). It is important to note that \( v_b \) does not necessarily corresponds to the trace of \( v_0 \) on \( \partial D \), even when the trace of \( v_0 \) exists.

\item \textbf{Weak Derivative:} For a weak function $v$, 
the weak derivative $\partial^w_{x_i} v$ of a weak function \( v \) with respect to the variable \( x_i \) is defined as a linear functional in the dual space \( H^{-1}(D) \), whose action on each \( q \in H^1(D) \) is given by
\begin{align}
    \int_D \partial^w_{x_i} v \, q \, dx &= -\int_D v_0 \, \partial_{x_i} q \, dx 
    + \int_{\partial D} v_b \, q \, n_i \, ds, \quad \forall~q \in H^1(D), \; 1 \leq i \leq 2, \tag{2.3}
\end{align}
where $\b{n} =  (n_1, n_2)$ is the outward unit normal vector on \( \partial D \).
\end{itemize}

With the foundational concepts of weak functions and their weak derivatives established, we now extend our focus to their discrete counterparts, which are crucial in formulating the modified weak finite element method. To introduce these discrete analogues with clarity and precision, we first present some prerequisite notations and preliminary results.

\par 
For a given mesh parameter \( h > 0 \), let \( \mathcal{T}_h \) denote the partition of \( \Omega \) into regular triangles, as described in \cite{Ciarlet:1978:FEM}, such that
\begin{align*}
\overline{\Omega} = \bigcup_{T \in \mathcal{T}_h} T.
\end{align*}
For each element \( T \in \mathcal{T}_h \), where \( r \) is a non-negative integer, let \( \mathbb{P}_r(T) \) represents the space of polynomials of degree at most \( r \) on \( T \). 

With respect to the triangulation \( \mathcal{T}_h \), we define the following sets:
\begin{itemize}[noitemsep]
    \item \( \mathcal{T}_p \): the collection of all elements sharing a common node \( p \),
    \item \( h_p \): the diameter of \( \mathcal{T}_p \),
    \item \( h : \max\{h_T : T \in \mathcal{T}_h\} \), where \( h_T \) is the diameter of element \( T \),
     \item \( h_e \): the length of edge \( e \),
    \item \( \mathcal{T}_e \): the set of triangles that share a common edge \( e \),
    \item \( \mathcal{V}_h \): the set of all vertices in the triangulation \( \mathcal{T}_h \), which is further categorized into:
    \begin{itemize}[noitemsep]
        \item \( \mathcal{V}_h^b \): boundary vertices lying on the boundary of \( \Omega \),
        \item \( \mathcal{V}_h^i \): interior vertices strictly inside the domain \( \Omega \).
    \end{itemize}
    \item \( \mathcal{V}_T \): the set of vertices of element \( T \),
    \item \( \mathcal{V}_e \): the set of vertices on edge \( e\),
    \item \( \mathcal{E}_h \): the set of all edges in \( \mathcal{T}_h \), which is decomposed as:
    \begin{itemize}[noitemsep]
        \item \( \mathcal{E}_h^i \): the set of interior edges of \( \mathcal{T}_h \),
        \item \( \mathcal{E}_h^b \): the set of boundary edges of \( \mathcal{T}_h \),
        \item \( \mathcal{E}_p \): the set of all edges sharing the node \( p \).
    \end{itemize}
\end{itemize}

The cardinality of any set \( S \) is denoted by \( |S| \), and the notation \( X' \) refers to the dual space of a Banach space \( X \). For a domain \( D \subset \Omega \) and an integer \( m \), the Sobolev space \( H^m(D) \) is equipped with the norm \( \|\cdot\|_{H^m(D)} \) and the seminorm \( |\cdot|_{H^m(D)} \). The \( L^2 \) inner product over \( D \) is denoted by \( (\cdot, \cdot)_{L^2(D)} \), where the subscript is omitted when \( D = \Omega \). Additionally, \( C \) represents a generic positive constant that is independent of the mesh size \( h \), and the notation \( X \lesssim Y \) indicates the existence of a positive constant \( C \) such that \( X \leq C Y \).
\par
\noindent
To develop the modified weak Galerkin method for the weak formulation \eqref{equation1_obs} in a systematic manner, we define the following broken Sobolev space:
\begin{align*}
{H^1}(\O, \mathcal{T}_h):= \{ {v} \in {L^2}(\O)|~ {v}_T:={v}|_{T} \in {H^1}(T)~\forall~T~\in~\mathcal{T}_h \},
\end{align*}
\par
\noindent
We now define some additional notations for the jumps and means of discontinuous functions, which will aid in their efficient representation across interfaces. For a scalar-valued function $w \in {H^1}(\O, \mathcal{T}_h) $ and vector-valued function $\b{v} \in [{H^1}(\O, \mathcal{T}_h)]^2$ which are double valued across the inter element boundary $e\in \mathcal{E}_h^i$,  the jumps and averages across the edge $e$ are defined as
\begin{align*}
\smean{w}&:= \dfrac{w|_{T_1} + w|_{T_2}}{2},~~~\sjump{w}:= w|_{T_1}\b{n_1} + w|_{T_2}\b{n_2} ; \\
\smean{\b{v}}&:= \dfrac{\b{v}|_{T_1} + \b{v}|_{T_2}}{2},~~~~\sjump{\b{v}}:= \b{v}|_{T_1}\cdot\b{n_1} + \b{v}|_{T_2}\cdot\b{n_2}; 
\end{align*}
where the edge $e$ is shared by two adjacent elements $T_1$ and $T_2$ and $\b{n_1}$ is the outward unit normal vector pointing from $T_1$ to $T_2$ with $\b{n_2} = -\b{n_1}$.  For the notational convenience,  we also define the jumps and averages on the boundary edge $e\in \mathcal{E}_h^b$ as
\begin{align*}
\smean{w}&:=w,~~~~~~~~~\sjump{w}:= w\b{n_e} ; \\
\smean{\b{v}}&:= \b{v},~~~~~~~~~~\sjump{\b{v}}:= \b{v}\cdot \b{n_e} ; \\
\end{align*}
where $T \in \mathcal{T}_h$ is such that $e \subset \partial T$ and $\b{n_e}$ is the unit normal on the edge $e$ pointing outside $T$. 
\par
\noindent
For the discretization, we introduce the following discrete finite element space to determine the discrete solution:
\begin{align*}
V_h^w &:= \left\{ v_h=(v_h^0, v_h^b)~:~ {v_h^0}|_T \in \mathbb{P}_1(T), v_h^b|_e = \smean{v_h^0}, e\subset \partial T,~ T \in \mathcal{T}_h \right\}, \\
V_h^{0w} &:= \left\{ v_h \in V_h^w,~v_h^b=0~\text{on}~\partial \Omega ~ \right\}.
\end{align*}
Incorporating the contact condition via integral constraints, we define the discrete analogue $\mathcal{K}_h$ of the non-empty, closed, convex set $\mathcal{K}$ as follows:
\begin{align*}
\mathcal{K}_h := \left\{ v_h \in  V_h^{0w} : \int\limits_T v^0_h~dx \geq \int\limits_T \psi~dx~~ \forall~T \in \mathcal{T}_h \right\}.
\end{align*}
It is important to note that for \( v_h = \{v_h^0, v_h^b\} \in V^w_h \), the component \( v_h^b = \smean{v_h^0} \) is not independent, as it is fully determined by \( v_h^0 \). Consequently, the weak finite element space \( V^w_h \) has fewer degrees of freedom, which reduces the number of unknowns in the corresponding weak finite element equation.
\subsection{\textbf{Discrete Weak Gradient}}
In this subsection, we begin by defining the discrete weak gradient used in the MWG method, and then proceed to discuss some important properties of this gradient operator.
\par 
\noindent
\textit{
\textbf{Definition 1} For a given partition $\mathcal{T}_h$ of $\Omega$ and a piecewise smooth function $v$ of~$\Omega,$ the discrete gradient of $v$ on $T$ is a unique element in $\nabla_{w,T}v \in [\mathbb{P}_0(T)]^2$ such that 
\begin{align}\label{weak_gradient}
\int\limits_T \nabla_{w,T} v \cdot \b{q}~dx = −\int\limits_T v~\nabla\cdot \b{q}~dx + \int\limits_{\partial T} \smean{v}\b{q} \cdot \b{n}~ds~~\forall~\b{q} \in [\mathbb{P}_{0}(T)]^2,
\end{align}
which in turn reduces to 
\begin{align}
\int\limits_T \nabla_{w,T} v \cdot \b{q}~dx =  \int\limits_{\partial T} \smean{v}\b{q} \cdot \b{n}~ds~~\forall~\b{q} \in [\mathbb{P}_{0}(T)]^2.
\end{align}}
With slight abuse of notation, we shall denote weak gradient $\nabla_{w,T}v$ as $\nabla_{w}v$ when no ambiguity arises.
\par 
\noindent
Note that when \( v \) is continuous in \( \Omega \), it follows that \( \smean{v} = v \) on \( \partial T \) for all \( T \in \mathcal{T}_h \). Using the definition of \( \nabla_wv \),  we observe
\begin{align*}
\int\limits_T \nabla_{w} v \cdot \b{q}~dx = −\int\limits_T v~\nabla\cdot \b{q}~dx + \int\limits_{\partial T}{v}\b{q} \cdot \b{n}~ds = \int\limits_T \nabla v \cdot \b{q}~dx, \quad \forall~\b{q} \in [\mathbb{P}_{0}(T)]^2,
\end{align*}
which demonstrates that the weak gradient \( \nabla_w v \) coincides with the \( L^2 \)-projection of the classical gradient \( \nabla v \) for such \( v \). Furthermore, if \( v \in \mathbb{P}_k(\Omega) \), we conclude that \( \nabla_w v = \nabla v \).
\par 
\noindent
In the subsequent analysis, we will utilize the following identity:
\begin{align*}
\sum_{T \in \mathcal{T}_h} \int\limits_{\partial T} w \b{q}\cdot\b{n}~ds = \sum_{e \in \mathcal{E}_h^i} \int\limits_e \smean{w}\sjump{\b{q}}~ds + \sum_{e \in \mathcal{E}_h} \int\limits_e \sjump{w}\cdot\smean{\b{q}}~ds. 
\end{align*}
The following lemma establishes an error bound between the discrete weak gradient and the classical gradient for $v_h=(v_h^0, v_h^b) \in V_h^{w}$ which becomes a handy tool in this analysis. 
\begin{lemma}\label{discrete_error}
For any $v_h=(v_h^0, v_h^b) \in V_h^{w},$ the following relation holds
\begin{align*}
\sum_{T \in \mathcal{T}_h} \|\nabla_w v_h - \nabla v_h^0\|^2_{L^2(T)} \leq \sum_{e \in \mathcal{E}_h} \frac{1}{h_e}\int_e \sjump{v^0_h}^2~ds. 
\end{align*}
\end{lemma}
\begin{proof}
Consider 
\begin{align*}
\sum_{T \in \mathcal{T}_h} \|\nabla_w v_h - \nabla v_h^0\|^2_{L^2(T)} =\sum_{T \in \mathcal{T}_h} \int\limits_T  (\nabla_w v_h - \nabla v_h^0)\cdot (\nabla_w v_h - \nabla v_h^0) ~dx
\end{align*}
Let $\b{q}=\nabla_w v_h - \nabla v_h^0.$ Thus, we have 
\begin{align}\label{2.2}
\sum_{T \in \mathcal{T}_h} \|\nabla_w v_h - \nabla v_h^0\|^2_{L^2(T)} &= \sum_{T \in \mathcal{T}_h} \int\limits_T  (\nabla_w v_h\cdot \b{q} - \nabla v_h^0\cdot \b{q})~dx.
\end{align}
Exploiting the definition of discrete weak gradient \eqref{weak_gradient} in \eqref{2.2}, we obtain
\begin{align}\label{rel_11}
\sum_{T \in \mathcal{T}_h} \|\b{q}\|^2_{L^2(T)} = - \sum_{T \in \mathcal{T}_h} \int\limits_{T} v_h^0(\nabla\cdot \b{q})~dx + \sum_{T \in \mathcal{T}_h} \int\limits_{\partial T} \smean{v_h^0}\b{q}\cdot \b{n}~ds -  \sum_{T \in \mathcal{T}_h}\int\limits_T \nabla v_h^0\cdot \b{q}~dx. 
\end{align} 
Using integration by parts on the first term on the right-hand side of $\eqref{rel_11}$, we derive
\begin{align*}
\sum_{T \in \mathcal{T}_h} \|\b{q}\|^2_{L^2(T)} &= \sum_{T \in \mathcal{T}_h} \int\limits_{\partial T} \bigg(\smean{v_h^0} - v_h^0 \bigg) \b{q}\cdot \b{n}~ds \nonumber\\
&= \sum_{e \in \mathcal{E}_h}\int\limits_e \sjump{\smean{v_h^0} - v_h^0}\cdot{\smean{\b{q}}}~ds + \sum_{e \in \mathcal{E}^i_h} \int\limits_e \smean{\smean{v_h^0} - v_h^0}{\sjump{\b{q}}}~ds.
\end{align*} 
Thus, it is easy to realize that 
\begin{align}\label{rel_1}
\sum_{T \in \mathcal{T}_h} \|\b{q}\|^2_{L^2(T)} &=  -\sum_{e \in \mathcal{E}_h} \int\limits_e \sjump{v_h^0}\cdot{\smean{\b{q}}}~ds.
\end{align} 
A use of Cauchy Schwarz inequality followed by inverse estimates (stated in equation \eqref{inverse 1}), we obtain
\begin{align*}
\sum_{T \in \mathcal{T}_h} \|\b{q}\|^2_{L^2(T)} &\leq \bigg(  \sum_{e \in \mathcal{E}_h} \|\sjump{v_h^0}\|^2_{L^2(e)} \bigg)^{\frac{1}{2}}  \bigg(  \sum_{e \in \mathcal{E}_h} \|\smean{\b{q}}\|^2_{L^2(e)} \bigg)^{\frac{1}{2}} \\
&\leq \bigg(  \sum_{e \in \mathcal{E}_h} h_e^{-1}\|\sjump{v_h^0}\|^2_{L^2(e)} \bigg)^{\frac{1}{2}}  \bigg(  \sum_{e \in \mathcal{E}_h} \sum_{T \in \mathcal{T}_e} \|{\b{q}}\|^2_{L^2(T)} \bigg)^{\frac{1}{2}} \\
&\leq \bigg(  \sum_{e \in \mathcal{E}_h} h_e^{-1}\|\sjump{v_h^0}\|^2_{L^2(e)} \bigg)^{\frac{1}{2}}  \bigg(\sum_{T \in \mathcal{T}_h} \|{\b{q}}\|^2_{L^2(T)}\bigg)^{\frac{1}{2}}.
\end{align*}
This completes the proof.
\end{proof}
With the necessary notations and results established, we now formulate the discrete problem corresponding to \eqref{equation1_obs}.
\subsection{Discrete problem}
The discrete weak formulation for the boundary value problem \eqref{equation1_obs} reads as: find $u_h \in \mathcal{K}_h$ such that 
\begin{align}\label{discrete}
a_h(u_h, v_h-u_h) \geq L(v_h-u_h) \quad \forall~v_h \in \mathcal{K}_h,
\end{align}
where,
\begin{align*}
a_h(u_h, v_h) &:= \sum_{T \in \mathcal{T}_h}\int\limits_T \nabla_w u_h \cdot \nabla_w v_h~dx + \sum_{e\in \mathcal{E}_h} \frac{1}{h_e} \int\limits_e \sjump{u^0_h} \cdot \sjump{v^0_h}~ds, \\
L(v_h-u_h) &:= \sum_{T \in \mathcal{T}_h}\int\limits_T f  (v_h-u_h)~dx.
\end{align*}
Here,  the bilinear form $a_h(\cdot, \cdot)$ is defined in such a way that the first term  is analogous to the bilinear form in the continuous setting,  while the second term incorporates the stability term. For convenience, we shall denote this stabilization term as $s_h(\cdot, \cdot)$.
\par 
\noindent
Next, we define the energy norm on $V^{0w}_h$ in which the error will be measured. To this end, we introduce a mesh-dependent norm on $V^{0w}_h$ which is well-defined for $H^1_0(\Omega) + V^{0w}_h$ as follows:
\begin{align*}
\norm{v_h}^2:=a_h(v_h,v_h).
\end{align*}
To demonstrate the well-posedness of the discrete problem \eqref{discrete}, we verify the coercivity and boundedness of the discrete bilinear form $a_h(\cdot,\cdot)$ with respect to energy norm $\norm{\cdot}$. Consequently, by applying the Lions and Stampacchia theorem \cite{Ciarlet:1978:FEM}, we conclude that a unique solution exists.
\par 
\noindent
Next, we state few technical lemmas \cite{Ciarlet:1978:FEM,BScott:2008:FEM} which will play a crucial role in the forthcoming convergence analysis ahead.

\begin{itemize}
\item \textbf{Trace inequality}: For each \(T \in \mathcal{T}_h\) and \(e \subset \partial T\), the following inequality holds for all \(v \in H^1(T)\):
\end{itemize}
\begin{align}\label{trace}
\|v\|^2_{L^2(e)} &\lesssim h^{-1}_e \|v\|^2_{L^2(T)} + h_e |v|^2_{H^1(T)}.
\end{align}

\begin{itemize}
\item \textbf{Inverse inequalities}: For each \(T \in \mathcal{T}_h\) and \(e \subset \partial T\), the following inequalities hold:
\begin{align}
\lVert v_h \rVert_{L^2(e)} &\lesssim h_e^{-\frac{1}{2}} \lVert v_h \rVert_{L^2(T)}, \label{inverse 1} \\
\lVert \nabla v_h \rVert_{L^2(T)} &\lesssim h_T^{-1} \lVert v_h \rVert_{L^2(T)},\label{inverse 2}
\end{align}
 for any \(v_h\) in $\mathbb{P}^k(T)$.
 \end{itemize}

\section{\textbf{Langrange Multilpiers}}
In this section, we define the continuous and discrete Lagrange multipliers to capture the displacement residuals in the continuous \eqref{equation1_obs} and discrete \eqref{discrete} variational inequalities, respectively. Prior to this, we shall introduce couple of interpolation operators. Since the constraints are formulated as integral conditions, special attention is given to incorporate this aspect while defining the Lagrange multipliers.
\subsection{Interpolation Operators}
We begin by defining the space of constant functions:
\begin{align*}
W_h:= \{  v \in L^1(\Omega): v_T := v|_{T} \in \mathbb{P}_0(T) \quad \forall~T \in \cT_h \}.
\end{align*}
In the context of the obstacle problem, contact arises within the interior of the domain. Consequently, the Lagrange multipliers must be defined over the entire domain $\Omega$ to effectively account for contact occurring within the body. To achieve this, we define the following projection operator $\pi_h^0: V_h^w \longrightarrow W_h$ to construct these multipliers. Define $\pi^0_hv_h:=\pi^0_hv_h|_{T},~T \in \mathcal{T}_h$ as 
\begin{align*}
\pi^0_hv_h|_{T} := \frac{1}{meas(T)}\int\limits_Tv_h^0~dx, ~~v_h=(v_h^0,v_h^b)\in V^w_h.
\end{align*}
We now proceed to to establish a fundamental property of the map $\pi^0_h,$ which plays a crucial role in the forthcoming analysis.
\begin{lemma} \label{lemma1}
The map $\pi_h^0: V_h^{0w} \longrightarrow W_h$ is onto.
\end{lemma}
\begin{proof}
Let $\mathcal{O}:=\big\{T^j,~j \in \{1,2, \cdots, NrElem\}\big\}$ represents the enumeration of triangles in $\mathcal{T}_h$. Resorting to the definition of the space ${{W}_h}$,  for any ${w}_h \in {{W}_h}$, we have ${w}_h|_{T^j} := \alpha^{j}$ for $1\leq j \leq NrElem$,  where $\alpha^{j}$ is a constant.
\par 
\noindent 
Next, we construct a test function $q_h = (q_h^0, q_h^b) \in V_h^{0w}$ such that for each $T^j \in \mathcal{T}_h$, we have:
\begin{align*}
q_h^0 = \alpha^j,~\text{and}~
q_h^b|_e = 
\begin{cases}
\smean{q_h^0}, & \text{if $e$ is an interior edge of } T^j, \\
0, & \text{if $e$ is a boundary edge of } T^j.
\end{cases}
\end{align*}
As a result, we have  $\pi_h^0(q_h)=w_h$, thereby proving the claim.
\end{proof}
The surjectivity of the map $\pi_h^0$ ensures a continuous right inverse $\pi_h^{-1}:W_h \longrightarrow V_h^{0w}$ which is defined as $${\pi}_h^{-1}({w}_h)={q}_h,$$ where ${q}_h$ is stated in the proof of Lemma \ref{lemma1}. 
\par 
\noindent
Next, we revisit the classical Cl$\acute{e}$ment interpolation result \cite[Section 4.8, Pg. 122]{BScott:2008:FEM}, which is crucial in establishing the upper bound.  

\begin{lemma}\label{clement}  
Let ${v} \in {V}$. Then, there exists ${v_h} \in H^1_0(\Omega) \cap V_h^w$ such that for any triangle \( T \in \mathcal{T}_h \), the following inequality holds:  
\begin{align*}
\|{v} - {v_h}\|_{{H^l}(T)} \lesssim h_T^{1-l} \|{v}\|_{{H^1}(\mathcal{T}_T)}, \quad l = 0, 1,
\end{align*}
where \(\mathcal{T}_T\) is the set of all elements \(\tilde{T}\) such that \(\tilde{T} \cap T \neq \emptyset\).  
\end{lemma}

\subsection{Lagrange Multipliers} We are now prepared to define the continuous and discrete Lagrange multipliers, followed by a discussion of their key properties.
\begin{itemize}
\item \textit{Continuous Lagrange Multiplier:} Define the Lagrange multiplier $\lambda(u) \in H^{-1}(\Omega)$ as
\begin{align} \label{eq:Lambda}
(\lambda(u), v)_{-1,1}:=L(v)- a(u,v) \quad \forall~v \in V,
\end{align}
where $(\cdot, \cdot)_{-1,1}$ denotes the duality pairing between $H^{-1}(\Omega)$ and $H^{1}_0(\Omega).$
In view of \eqref{equation1_obs} and \eqref{eq:Lambda}, we have
\begin{equation} \label{eq:Prop1}
(\lambda(u), v-u)_{-1,1} \leq 0  \quad \forall~v \in \mathcal{K}. 
\end{equation}
By choosing \( v = u + \phi \) with \( \phi \in H^1_0(\Omega) \) and \( \phi \geq 0 \) in \eqref{eq:Prop1} , we obtain \( \lambda(u) \leq 0 \). 
In this analysis, \( f \in L^2(\Omega) \), \( \psi \in H^2(\Omega) \), and \( \Omega \) is assumed to be convex. convex. Under these conditions, regularity theory for the obstacle problem guarantees that the solution \( u \in H^2(\Omega) \) (see \cite{AH:2009:VI}). Furthermore, the Lagrange multiplier \( \lambda(u) \), defined in \eqref{eq:Lambda}, can be written as \( \lambda(u) = f + \Delta u \), which ensures that \(\lambda(u) \in L^2(\Omega) \).
\par 
\noindent
The following lemma is a consequence of \eqref{eq:Lambda} and \eqref{eq:Prop1}. (see \cite{Glowinski:2008:VI,AH:2009:VI}).
\begin{lemma}
If ${u}\in H^2(\Omega),$ then we have $\lambda(u) \in L^2(\Omega)$ with $\lambda(u) \leq 0$ almost everywhere in $\Omega.$ Furthermore,
\begin{align*}
(\lambda(u),  u - \psi)_{-1,1} = 0.
\end{align*}
Additionally,  if $u > \psi$ in any open set $\Omega^{'} \subset \Omega$ then $\lambda(u) \equiv 0$ in $\Omega^{'}.$
\end{lemma}
\par 
\noindent
Next,  we define the discrete analogue of continuous Lagrange multiplier $\lambda(u)$ for the obstacle problem. \item \textit{Discrete Lagrange Multiplier:} 
With the assistance of map ${\pi}_h^{-1}$, we introduce the discrete Lagrange multiplier \( \lambda_h(u_h) \in W_h \), which can be considered as an approximation to the functional \( \lambda(u) \),  as:
\begin{align*}
( \lambda_h(u_h), w_h)_{L^2(\Omega)} := L({\pi}_h^{-1}({w}_h)) - a_h(u_h,{\pi}_h^{-1}({w}_h))\quad \forall~ w_h \in W_h.
\end{align*} 
The following lemma guarantees the well-definedness of this map and establishes a crucial property of the discrete Lagrange multiplier, which is widely used in the subsequent analysis.
\begin{lemma} \label{lang_prop}
The following holds:
\begin{itemize}
\item The map $\lambda_h(u_h) \in W_h$ is well-defined.
\item Furthermore, for any $v_h \in V^{0w}_h$,  $(\lambda_h(u_h),v_h)_{L^2(\Omega)}=L(v_h)-a_h(u_h,v_h)$.
\end{itemize}
\end{lemma}
\begin{proof}
We will prove the two arguments one by one.
\paragraph{1. Well-definedness of the map \(\lambda_h(u_h)\):} Let \( q_h^1, q_h^2 \in W_h \) be such that \( q_h^1 = q_h^2 \). By leveraging the surjectivity of the map \( \pi^0_h \), there exist \( v_h^1, v_h^2 \in V^{0w}_h \) satisfying \( v_h^1 = \pi_h^{-1} q_h^1 \) with \( \pi^0_h v_h^1 = q_h^1 \) and \( v_h^2 = \pi_h^{-1} q_h^2 \) with \( \pi^0_h v_h^2 = q_h^2 \). Define \( v_h := v_h^1 - v_h^2 \). Then, \( \pi^0_h(v_h) = \pi^0_h(v_h^1 - v_h^2) = q_h^1 - q_h^2 = 0 \). Next, for any \( v_h \in V^{w}_h \) such that \( \pi^0_h(v_h)|_T = 0 \) for all \( T \in \mathcal{T}_h \), the following identity holds:
\begin{align}\label{rel2_new}
a_h(u_h, v_h) - L(v_h) = 0.
\end{align}
This follows by considering \( p_h = u_h \pm v_h \in \mathcal{K}_h \) and substituting into the discrete equation, leading to relation \eqref{rel2_new}. Thus, \( \lambda_h(u_h) \) is well-defined.
\\
\paragraph{2. Proving the identity $(\lambda_h(u_h),v_h)_{L^2(\Omega)}=L(v_h)-a_h(u_h,v_h)$ for any $v_h \in V^{0w}_h$:} 
 Using the definition of \( \lambda_h(u_h) \), we compute:
\begin{align*}
(\lambda_h(u_h), v_h)_{L^2(\Omega)} &= (\lambda_h(u_h), \pi^0_h v_h)_{L^2(\Omega)} \\
&= L(\pi_h^{-1} \pi^0_h v_h) - a_h(u_h, \pi_h^{-1} \pi^0_h v_h) \\
&= L(\pi_h^{-1} \pi^0_h v_h - v_h) - a_h(u_h, \pi_h^{-1} \pi^0_h v_h - v_h) \\
&\quad + L(v_h) - a_h(u_h, v_h) \\
&= L(v_h) - a_h(u_h, v_h),
\end{align*}
since \( \pi^0_h (\pi_h^{-1} \pi^0_h v_h - v_h) = 0 \). This completes the proof.\end{proof}

\begin{remark}\label{remark 1}
 We note that $\lambda_h(u_h) \in W_h$ can be identified with the functional on $V$ which we denote by $\lambda_h(u_h)$ itself using the following relation $ \left  (\lambda_h(u_h),v \right)_{-1,1}=(\lambda_h(u_h),v) ~~\forall~ v \in V.$
\end{remark}

We now proceed by categorizing the elements of the triangulation $\mathcal{T}_h$ into two distinct sets based on the discrete constraint in problem \eqref{discrete}. To this end, we define the following two sets:
\begin{align}
\mathcal{C}_h &:= \{ T \in \cT_h| ~\pi^0_h|_T(u_h)=\pi^0_h|_T(\psi)  \}, \label{eq:contact}\\
\mathcal{N}_h &:= \{ T \in \cT_h |~ \pi^0_h|_T(u_h) > \pi^0_h|_T(\psi) \}. \label{eq:NonC}
\end{align}
In the following lemma, we establish the sign relations satisfied by the discrete Lagrange multiplier $\lambda_h(u_h).$

\begin{lemma} \label{sign_langrange}
Let $T \in \mathcal{T}_h$.  Then,  the discrete Lagrange multiplier $\lambda_h(u_h)$ satisfies 
\begin{align*}
 \lambda_h(u_h) \leq 0~\text{on}~\bar \Omega.
\end{align*}
Moreover, for any triangle $T$ in non-contact set $\mathcal{N}_h$,   we have $ \lambda_h(u_h) = 0.$  
\end{lemma}
\begin{proof}
Let $z_h \in W_h$ such that $z_h|_T \geq 0~\forall~T \in \mathcal{T}_h$. The first identity can realised by taking the test function $v_h= u_h + w_h$ in \eqref{discrete} where $w_h \in V^{0w}_h$ such that $z_h=\pi^0_h(w_h).$ Thus, we get $a_h(u_h,w_h)\geq L(w_h).$ Hence, $\lambda_h(u_h) \leq 0.$
\par 
\noindent 
Recalling the definition of $\mathcal{N}_h$, we have  $\pi^0_h(u_h)> \pi^0_h(\psi)$ for all $T \in \mathcal{N}_h$.  For any $T \in \mathcal{N}_h^O$, let $p \in \cV_T$ and $\phi_h^p $ be the nodal basis function corresponding to vertex $p$. For sufficiently small $\epsilon>0$, we note that 
$v_h =u_h \pm \epsilon \phi_h^p \in \cK_h$. Taking these choices of $v_h$ in \eqref{discrete}, we find $a_h(u_h, \phi_h^p)= L(\phi_h^p)$  and hence $\lambda_h(u_h)=0 ~ \text{on}~ T$.
\end{proof}
\end{itemize}

In the upcoming section,  we aim to derive reliable and efficient \textit{a posteriori} error estimates for MWG finite element method applied to obstacle problem.
\section{\textbf{Aposteriori Error Analysis}}
This section aims to derive \textit{a posteriori} error estimates, which are essential for identifying regions of irregular solution behaviour. Unlike \textit{a priori} error estimates, \textit{a posteriori} error estimates do not require prior assumptions about the solution's regularity, making them more versatile in practical applications. Ensuring the reliability and efficiency of the proposed error estimators is essential. Reliability implies that the error estimators establishes an upper bound for the true error, ensuring accurate assessment of the error throughout the solution domain. In the context of adaptive mesh refinement, a reliable \textit{a posteriori} error estimator is crucial for effectively identifying regions that require refinement. However, to prevent excessive overestimation of the error and to achieve high accuracy with minimal computational effort, the error estimator must also provide a lower bound, thereby ensuring efficiency.
\par 
\noindent
We begin by introducing additional functional tools essential for the subsequent analysis.
\subsection{Space Decomposition Technique} 
We introduce the discontinuous and conforming finite element spaces, denoted by 
\( V^D_h \) and  \( V^C_h \), respectively which play a key role in the space decomposition technique. These spaces are defined as
\begin{align*}
V^D_h &:= \bigg\{ v_h \in L^2(\Omega) : v_h|_{{T}} \in \mathbb{P}_1({T}), \quad \forall~T \in \mathcal{T}_h \bigg\}, \\
V^C_h &:= \bigg\{ v_h \in H^1_0(\Omega) : v_h|_{{T}} \in \mathbb{P}_1({T}), \quad \forall~T \in \mathcal{T}_h \bigg\}.
\end{align*}
With these definitions in place, we proceed to apply the space decomposition technique for the discontinuous finite element space \( V^D_h \). Specifically, ${{V}^D_h}$ is decomposed into the conforming subspace ${{V}^C_h}$ and non-conforming subspace $({V}^C_h)^{\perp}$, which is defined as the orthogonal complement of ${V}^C_h$ in ${{V}^D_h}$ with the following inner product $( \cdot, \cdot )_{1,h}$:
\begin{align*}
( q_h, p_h )_{1,h}:=(\nabla q_h, \nabla p_h)_{L^2(\Omega)} + \sum_{e \in \mathcal{E}_h} \frac{1}{h_e} \int\limits_e \sjump{q_h}\cdot \sjump{p_h}~ds.
\end{align*}
where $q_h,~p_h \in {V}^D_h$. 
\par 
\noindent
In the following lemma, we present the important approximation properties discussed in the article \cite{Karakashian: 2007: modified}.
\begin{lemma}\label{decom}
For any $w_h \in V_h^D$, there exists a unique decomposition $w_h = w_h^c+(w_h^c)^{\perp}$ where $w_h^c \in {V}^C_h$ and $(w_h^c)^{\perp} \in ({V}^C_h)^{\perp}$ such that it satisfies the following approximation properties:
\begin{align*}
\sum_{T \in \mathcal{T}_h}\|\nabla{(w_h^c)^{\perp}}\|^2_{L^2(T)} &\lesssim \sum_{e \in \mathcal{E}_h} \int\limits_e |\sjump{w_h}|^2~ds, \\
\sum_{T \in \mathcal{T}_h}\|h_T^{-1}(w_h^c)^{\perp}\|^2_{L^2(T)} &\lesssim \sum_{e \in \mathcal{E}_h} \int\limits_e |\sjump{w_h}|^2~ds. \\
\end{align*}
\end{lemma}
\begin{remark}\label{remark11}
It is straightforward to observe that any function \( v_h \in V_h^C \), can also be expressed as an element of \( V_h^{0w} \) with the components \( v_h^0 = v_h|_T \) and \( v_h^b = \smean{v_h^0} = v_h|_{\partial T} \). Consequently, we obtain the inclusion  \( V_h^C \subset V_h^{0w} \).
\end{remark}
\par 
\noindent
The primary goal of this article is to derive upper and lower bound estimates for the following quantity 
\begin{align*}
\norm{ u_h - u}^2 + \sum_{ T \in \mathcal{T}_h}\|\lambda(u) - \lambda_h(u_h)\|_{H^{-1}(T)}.
\end{align*}
In the following sections, we derive \textit{a posteriori} error estimates and discuss the reliability and efficiency of the proposed error estimator.
 
\subsection{Reliable Error Estimates}
We begin by introducing the following residual error estimators:
\begin{align*}
\eta^2_1 &:= \sum_{T \in \mathcal{T}_h}h_T^2 \|f + \nabla \cdot(\nabla_wu_h)-\lambda_h(u_h)\|^2_{L^2(T)}, \\
\eta^2_2 &:= \sum_{e\in \mathcal{E}_h^i}h_e\|\sjump{\nabla_wu_h}\|^2_{L^2(e)},\\
\eta^3_2 &:= \sum_{e\in \mathcal{E}_h}\frac{1}{h_e}\|\sjump{u^0_h}\|^2_{L^2(e)}.
\end{align*}
The total residual error estimator $\eta_h$ is defined as
\begin{equation}\label{eq:Est1}
\eta_h^2:= \eta_1^2+\eta_2^2+\eta_3^2+ \|\nabla (\psi-(u_h^0)^c)^+\|^2_{L^2(\Omega)} - \sum_{T \in \mathcal{C}_h}  \int_T \lambda_h(u_h) ((u_h^0)^c - \psi)^{+} ~dx.
\end{equation}
Next, we introduce a continuous linear functional $\mathcal{G}_h$ on $H^{1}_0(\Omega)$ which provides a framework for effectively capturing and analyzing the errors in the system. Specifically, we define $\mathcal{G}_h: H^{-1}(\Omega) \longrightarrow \mathbb{R}$ as 
\begin{align}\label{abstract:Galerkin}
{\mathcal{G}_h}({v}) := a_h({u - u_h}, {v}) + ( \lambda(u)-\lambda_h(u_h), ~v )_{-1,1}  \quad \forall~{v} \in {H^1_0(\Omega)}.
\end{align}

\par 
\noindent
We begin by examining the error bounds for both the displacement term and the Lagrange multiplier, using the dual norm of the functional $\mathcal{G}_h$ and the duality pairing between the displacements and the Lagrange multiplier, as presented in the following lemma.
\begin{lemma}
Let $u$ and $u_h = (u^0_h,u^b_h)$ be the solution of continuous and discrete variational inequalities $\eqref{equation1_obs}$ and $\eqref{discrete}$, respectively. It holds that 
\begin{align*}
\norm{ u_h - u}^2 + \sum_{ T \in \mathcal{T}_h}\|\lambda(u) - \lambda_h(u_h)\|^2_{H^{-1}(T)} \leq C_1\|\mathcal{G}_h\|^2_{H^{-1}(\Omega)} + \eta_3^2 + C_2(\lambda_h(u_h)-\lambda(u), u-(u_h^0)^c)_{-1,1}
\end{align*}
where $C_1,~C_2$ are positive constants independent of mesh parameter $h$ and $u_h^0 = (u_h^0)^c + ((u_h^0)^c)^{\perp}$ is defined in Lemma $\ref{decom}.$
\vspace{-0.3 cm}
\begin{proof}
Consider 
\begin{align}\label{eqn111}
\norm{ u_h - u}^2 \leq \norm{ u_h - u_h^0}^2 + \norm{ u_h^0 - u}^2.
\end{align}
The bound on the first term on the right-hand side of the previous relation is obtained from Lemma \ref{discrete_error} as follows:
\begin{align}\label{eqn21}
\norm{ u_h - u_h^0}^2 = \sum_{ T \in \mathcal{T}_h}\|\nabla_w u_h -\nabla u_h^0\|^2_{L^2(T)} \leq \eta_3^2.
\end{align}
Next, we turn our attention to the second term in equation $\eqref{eqn111}$. Since $u_h^0 \in V_h^D,$ by applying the decomposition  $u_h^0 = (u_h^0)^c + ((u_h^0)^c)^{\perp},$ as defined in Lemma $\ref{decom}$, we obtain
\begin{align*}
 \norm{ u_h^0 - u}^2 &\leq \eta_3^2 + \sum_{ T \in \mathcal{T}_h}\|\nabla u_h^0 -\nabla u\|^2_{L^2(T)} \\&\leq \sum_{ T \in \mathcal{T}_h}\|\nabla (u_h^0)^c -\nabla u\|^2_{L^2(T)} + \sum_{ T \in \mathcal{T}_h}\|\nabla ((u_h^0)^c)^{\perp}\|^2_{L^2(T)} +\eta_3^2.
\end{align*}
Incorporating the bounds stated in Lemma $\ref{decom}$, we obtain
\begin{align}\label{eqn_new}
\norm{ u_h^0 - u}^2 \lesssim \sum_{ T \in \mathcal{T}_h}\|\nabla (u_h^0)^c -\nabla u\|^2_{L^2(T)} + \eta_3^2.
\end{align}
Next, we focus on simplifying the term $\sum\limits_{ T \in \mathcal{T}_h}\|\nabla (u_h^0)^c -\nabla u\|^2_{L^2(T)}$ as follows:
\begin{align*}
\sum_{ T \in \mathcal{T}_h}\|\nabla &(u_h^0)^c -\nabla u\|^2_{L^2(T)} \\&= \sum_{ T \in \mathcal{T}_h} \int\limits_T (\nabla (u_h^0)^c -\nabla u) \cdot (\nabla (u_h^0)^c -\nabla u)~dx \\
&= \sum_{ T \in \mathcal{T}_h} \int\limits_T (\nabla (u_h^0)^c -\nabla_w u_h) \cdot (\nabla (u_h^0)^c -\nabla u)~dx + \sum_{ T \in \mathcal{T}_h} \int\limits_T (\nabla_w u_h -\nabla u) \cdot (\nabla (u_h^0)^c -\nabla u)~dx \\
&=\sum_{ T \in \mathcal{T}_h} \int\limits_T (\nabla (u_h^0)^c -\nabla u_h^0) \cdot (\nabla (u_h^0)^c -\nabla u)~dx + \sum_{ T \in \mathcal{T}_h} \int\limits_T (\nabla u_h^0 -\nabla_w u_h) \cdot (\nabla (u_h^0)^c -\nabla u)~dx  \\&+ \sum_{ T \in \mathcal{T}_h} \int\limits_T (\nabla_w u_h -\nabla u) \cdot (\nabla (u_h^0)^c -\nabla u)~dx 
= Q_1 +Q_2 +Q_3,
\end{align*}
where,
\begin{align*}
Q_1&:=\sum_{ T \in \mathcal{T}_h} \int\limits_T (\nabla_w u_h -\nabla u) \cdot (\nabla (u_h^0)^c -\nabla u)~dx, \\
Q_2 &:= \sum_{ T \in \mathcal{T}_h} \int\limits_T (\nabla (u_h^0)^c -\nabla u_h^0) \cdot (\nabla (u_h^0)^c -\nabla u)~dx, \\
Q_3 &:= \sum_{ T \in \mathcal{T}_h} \int\limits_T (\nabla u_h^0 -\nabla_w u_h) \cdot (\nabla (u_h^0)^c -\nabla u)~dx.
\end{align*}
Next, we aim to estimate each of these terms individually. First, we focus on bounding $Q_2$ by applying Cauchy Schwarz inequality and utilizing Lemma $\ref{decom}$ as follows:
\begin{align*}
Q_2 &\leq \bigg(\sum_{ T \in \mathcal{T}_h} \int\limits_T (\nabla (u_h^0)^c -\nabla u_h^0)^2~dx \bigg)^{\frac{1}{2}}\bigg(\sum_{ T \in \mathcal{T}_h} \int\limits_T (\nabla (u_h^0)^c -\nabla u)^2~dx \bigg)^{\frac{1}{2}} \\
&\leq\frac{1}{\epsilon} \eta_3^2 + \epsilon \sum_{T\in \mathcal{T}_h}\| \nabla(u_h^0)^c -\nabla u \|^2_{L^2(T)},
\end{align*}
for some infinitesmall $\epsilon >0.$ On the similar lines, the bound on term $Q_3$ is derived by applying the Cauchy Schwarz inequality, followed by using the bounds from Lemma $\ref{discrete_error}$ as outlined below
\begin{align*}
Q_3 \leq \frac{1}{\epsilon_1} \eta_3^2 + \epsilon_1 \sum_{T\in \mathcal{T}_h}\| \nabla(u_h^0)^c -\nabla u \|^2_{L^2(T)},
\end{align*}
where $\epsilon_1>0$ is a small constant. Next, we focus on estimating $Q_1$. To this end, consider the following:
\begin{align*}
Q_1=\sum_{ T \in \mathcal{T}_h} \int\limits_T (\nabla_w u_h -\nabla u) \cdot (\nabla (u_h^0)^c -\nabla u)~dx = \sum_{ T \in \mathcal{T}_h} \int\limits_T (\nabla_w u_h -\nabla_w u) \cdot (\nabla_w (u_h^0)^c -\nabla_w u)~dx.
\end{align*}
Using the definition of Galerkin functional \eqref{abstract:Galerkin} and identity $ab\leq \frac{a^2}{2} + \frac{b^2}{2} $, we obtain 
\begin{align*}
Q_1&= -\mathcal{G}_h((u_h^0)^c - u) + ( \lambda_h(u_h) - \lambda(u), ~u -(u_h^0)^c )_{-1,1} \\
&\leq \frac{\| \mathcal{G}_h\|^2_{H^{-1}(\Omega)}}{2} + \frac{\|(u_h^0)^c - u\|^2_{H^1(\Omega)}}{2} + ( \lambda_h(u_h) - \lambda(u), ~   u - (u_h^0)^c)_{-1,1}. 
\end{align*}
By combining the individual estimates for $Q_1, ~Q_2$ and $Q_3$, we derive the following global bound:
\begin{align}\label{final}
\sum_{ T \in \mathcal{T}_h}\|\nabla (u_h^0)^c -\nabla u\|^2_{L^2(T)} &\leq \big(\frac{1}{\epsilon_1}+\frac{1}{\epsilon} \big) \eta_3^2 + \big(\epsilon_1 + \epsilon \big) \sum_{T\in \mathcal{T}_h}\| \nabla(u_h^0)^c -\nabla u \|^2_{L^2(T)} \nonumber \\ &+ \frac{\| \mathcal{G}_h\|^2_{H^{-1}(\Omega)}}{2} + \frac{\|(u_h^0)^c - u\|^2_{H^1(\Omega)}}{2} + ( \lambda_h(u_h) - \lambda(u), ~   u - (u_h^0)^c)_{-1,1}. 
\end{align}
The bounds established in equations \eqref{eqn111}, \eqref{eqn21}, \eqref{eqn_new} and \eqref{final}, provide us with a comprehensive bound for $\norm{u_h-u}$.
\vspace{0.3 cm}
\\
Subsequently, we proceed to bound $\sum\limits_{ T \in \mathcal{T}_h}\|\lambda(u) - \lambda_h(u_h)\|_{H^{-1}(T)}.$ Using the definition of the dual norm, we obtain the following expression:
\begin{align*}
\sum_{ T \in \mathcal{T}_h}\|\lambda(u) - \lambda_h(u_h)\|_{H^{-1}(T)} := \underset{\xi~\in~H^1_0(\Omega)}{Sup} \frac{ ( \lambda(u) - \lambda_h(u_h), \xi )_{-1,1}}{\|\xi\|_{H^1(\Omega)}}.
\end{align*}
The bound on $( \lambda(u) - \lambda_h(u_h), \xi )_{-1,1}$ can be realised using the definition of Galerkin functional \eqref{abstract:Galerkin} followed by Cauchy Schwarz inequality as detailed in the following steps
\begin{align*}
( \lambda(u) - \lambda_h(u_h), \xi )_{-1,1} &= \mathcal{G}_h(\xi) - a_h(u-u_h, \xi) \\
&\leq \|\mathcal{G}_h\|_{H^{-1}(\Omega)}\|\xi\|_{H^1(\Omega)} + \sum_{ T \in \mathcal{T}_h}\|\nabla_w u_h -\nabla_w u\|_{L^2(T)} \|\xi\|_{H^1(T)}.
\end{align*} 
Thus, we have 
\begin{align}\label{eqn_refnew2}
\|\lambda(u) - \lambda_h(u_h)\|^2_{H^{-1}(\Omega)}
\lesssim \|\mathcal{G}_h\|^2_{H^{-1}(\Omega)} + \sum_{ T \in \mathcal{T}_h}\|\nabla_w u_h -\nabla_w u\|^2_{L^2(T)}.
\end{align}
Finally the proof is completed utilizing the bound on $\sum_{ T \in \mathcal{T}_h}\|\nabla_w u_h -\nabla_w u\|^2_{L^2(T)}$ in relation \eqref{eqn_refnew2}.
\end{proof}
\end{lemma}
To obtain reliable estimates, we now focus on deriving bounds for \( \|\mathcal{G}_h\|_{H^{-1}(\Omega)} \) and \( (\lambda_h(u_h) - \lambda(u), u - (u_h^0)^c)_{-1,1} \), which will be addressed in the following two lemmas.
\begin{lemma}
It holds that 
\begin{align*}
\|\mathcal{G}_h\|^2_{H^{-1}(\Omega)} \lesssim  \eta_1^2 + \eta_2^2.
 \end{align*}
 \begin{proof}
 Let $\Phi \in H^1_0(\Omega).$ Then, there exists a Clement approximation of $\Phi$, denoted by $\Phi_c \in V_h^C$ (see Lemma \ref{clement}) such that 
 \begin{align}\label{appprox}
 \sum_{T \in \mathcal{T}_h} \bigg( h_T^{-2} \| \Phi -\Phi_c \|^2_{L^2(T)} + \|\nabla(\Phi -\Phi_c)\|^2_{L^2(T)} \bigg) \lesssim \|\nabla{\Phi}\|_{L^2(\Omega)}.
 \end{align}
Consider 
 \begin{align*}
\mathcal{G}_h(\Phi) = \mathcal{G}_h(\Phi - \Phi_c) + \mathcal{G}_h( \Phi_c). 
 \end{align*}
 We will estimate these terms individually. The bound on $\mathcal{G}_h(\Phi - \Phi_c) $ is derived using the following reasoning:
 \begin{align}\label{3.1}
 \mathcal{G}_h(\Phi - \Phi_c) &= a_h(u-u_h, \Phi - \Phi_c) +   ( \lambda(u)-\lambda_h(u_h), ~\Phi - \Phi_c
  )_{-1,1}  \nonumber\\
  & = \sum_{T \in \mathcal{T}_h}\int\limits_T \nabla_w(u-u_h)\cdot\nabla(\Phi - \Phi_c)~dx + \sum_{T \in \mathcal{T}_h}\int\limits_T (\lambda(u)-\lambda_h(u_h))(\Phi - \Phi_c)~dx\nonumber \\
  &= \sum_{T \in \mathcal{T}_h}\int\limits_T f(\Phi - \Phi_c)~dx - \sum_{T \in \mathcal{T}_h}\int\limits_T \nabla_w u_h\cdot\nabla(\Phi - \Phi_c)~dx - \sum_{T \in \mathcal{T}_h}\int\limits_T \lambda_h(u_h)(\Phi - \Phi_c)~dx,
 \end{align}
using the relation \eqref{eq:Lambda} in the last equation.  A use of integration by parts in the second term of \eqref{3.1}, we obtain 
 \begin{align*}
\mathcal{G}_h(\Phi - \Phi_c) &=  \sum_{T \in \mathcal{T}_h}\int\limits_T f(\Phi - \Phi_c)~dx + \sum_{T \in \mathcal{T}_h}\int\limits_T \nabla \cdot(\nabla_w u_h)(\Phi - \Phi_c)~dx - \sum_{T \in \mathcal{T}_h}\int\limits_{\partial T}(\nabla_w u_h\cdot \b{n})(\Phi - \Phi_c)~ds \\&- \sum_{T \in \mathcal{T}_h}\int\limits_T \lambda_h(u_h)(\Phi - \Phi_c)~dx \\
&= \sum_{T \in \mathcal{T}_h}\int\limits_T (f + \nabla \cdot(\nabla_w u_h) - \lambda_h(u_h) )(\Phi - \Phi_c)~dx - \sum_{T \in \mathcal{T}_h}\int\limits_{\partial T}(\nabla_w u_h\cdot \b{n})(\Phi - \Phi_c)~ds \\
&= \sum_{T \in \mathcal{T}_h}\int\limits_T (f + \nabla \cdot(\nabla_w u_h) - \lambda_h(u_h) )(\Phi - \Phi_c)~dx - \sum_{e \in \mathcal{E}^i_h} \int\limits_e \sjump{\nabla_w u_h}\smean{\Phi-\Phi_c}~ds \\&- \sum_{e \in \mathcal{E}_h} \int\limits_e \smean{\nabla_w u_h}\cdot\sjump{\Phi-\Phi_c}~ds. 
\end{align*}
Since on interior edges, $\sjump{\Phi-\Phi_c} = (0,0), \smean{\Phi-\Phi_c} = \Phi-\Phi_c$, and on boundary edges $\Phi-\Phi_c = 0$, we obtain
\begin{align*}
\mathcal{G}_h(\Phi - \Phi_c) &= \sum_{T \in \mathcal{T}_h}\int\limits_T (f + \nabla \cdot(\nabla_w u_h) - \lambda_h(u_h) )(\Phi - \Phi_c)~dx - \sum_{e \in \mathcal{E}^i_h} \int\limits_e \sjump{\nabla_w u_h}(\Phi-\Phi_c)~ds \\
&\leq \sum_{T \in \mathcal{T}_h}h_T \| f + \nabla \cdot(\nabla_w u_h) - \lambda_h(u_h) \|_{L^2(T)}h_T^{-1}\|\Phi - \Phi_c\|_{L^2(T)} \\&+ \sum_{e \in \mathcal{E}^i_h} h_e^{\frac{1}{2}}\|\sjump{\nabla_w u_h}\|_{L^2(e)}h_e^{-\frac{1}{2}}\|\Phi-\Phi_c\|_{L^2(e)}.
\end{align*}
Exploiting the approximation properties of the interpolation operator, discrete trace inequality together with discrete Cauchy Schwarz inequality, we obtain
\begin{align*}
\mathcal{G}_h(\Phi - \Phi_c) 
&\lesssim \Bigg(\bigg(\sum_{T \in \mathcal{T}_h}h^2_T \| f + \nabla \cdot(\nabla_w u_h) - \lambda_h(u_h) \|^2_{L^2(T)}\bigg)^{\frac{1}{2}} + \bigg( \sum_{e \in \mathcal{E}^i_h} h_e\|\sjump{\nabla_w u_h}\|^2_{L^2(e)}  \bigg)^{\frac{1}{2}}\Bigg) \|\nabla \Phi\|_{L^2(\Omega)} \\
&= \big( \eta_1 +\eta_2 \big)\|\nabla \Phi\|_{L^2(\Omega)}.
\end{align*}
Next, we estimate $\mathcal{G}_h( \Phi_c)$ in the following way 
\begin{align*}
\mathcal{G}_h( \Phi_c) &= \sum_{T \in \mathcal{T}_h}\int\limits_T\nabla_w(u-u_h) \cdot \nabla_w(\Phi_c)~dx + \sum_{T \in \mathcal{T}_h}\int\limits_T (\lambda(u)-\lambda_h(u_h))\Phi_c~dx  \\
&= \sum_{T \in \mathcal{T}_h}\int\limits_T f\Phi_c~dx -\int\limits_T\nabla_w u_h \cdot \nabla_w \Phi_c~dx- \sum_{T \in \mathcal{T}_h}\int\limits_T \lambda_h(u_h)\Phi_c~dx, 
\end{align*}
where the last equality follows from \eqref{eq:Lambda}. Further a use of Lemma \ref{lang_prop} and exploiting Remarks \ref{remark11} and \ref{remark 1}, we obtain 
\begin{align*}
\mathcal{G}_h( \Phi_c) &= (\lambda_h(u_h),\Phi_c)_{L^2(\Omega)} -\sum_{T \in \mathcal{T}_h}\int\limits_T \lambda_h(u_h)\Phi_c~dx =0.
\end{align*}
This completes the proof.
 \end{proof}
\end{lemma}
In the next lemma, we aim to estimate $(\lambda_h(u_h)-\lambda(u), u-(u_h^0)^c)_{-1,1}.$
\begin{lemma}
It holds that,
\begin{align*}
\hspace{ 2 cm}
( \lambda_h(u_h)-\lambda(u), u - (u_h^0)^c )_{-1,1} &\leq  \epsilon \| \lambda(u) -\lambda_h(u)\|_{H^{-1}(\Omega)}^2 + \frac{1}{\epsilon}  \|\nabla (\psi-(u_h^0)^c)^+\|^2_{L^2(\Omega)} \\ &-\sum_{T \in \mathcal{C}_h}  \int_T \lambda_h(u_h) ((u_h^0)^c - \psi)^{+} ~dx
\end{align*}
for a sufficiently small $\epsilon>0.$
\end{lemma}
\begin{proof} 
In the light of Remark \ref{remark11}, we find that $(u_h^0)^c \in V_h^{0w}.$ Utilizing this, we define the function $\tilde{u_h}:= max((u_h^0)^c, \psi)$ which belongs to $\mathcal{K}.$ Using equation \eqref{eq:Prop1}, we ensure that $ ( \lambda(u), \tilde{u_h} - u )_{-1,1} \leq 0.$  Further, taking into the account that $\lambda_h(u_h) \leq 0$ and $u \geq \psi$, we have
\begin{align*}
( \lambda_h(u_h)-\lambda(u), u &- (u_h^0)^c )_{-1,1}\\ &= 
( \lambda_h(u_h), u - (u_h^0)^c )_{-1,1} -( \lambda(u), u - \tilde{u_h})_{-1,1} - ( \lambda(u), \tilde{u_h} - (u_h^0)^c )_{-1,1} \\
&\leq ( \lambda_h(u_h), u - (u_h^0)^c )_{-1,1} - ( \lambda(u), \tilde{u_h} - (u_h^0)^c )_{-1,1} \\
&= ( \lambda_h(u_h), u - (u_h^0)^c - \tilde{u_h} + (u_h^0)^c )_{-1,1} - ( \lambda(u) - \lambda_h(u_h), \tilde{u_h} - (u_h^0)^c )_{-1,1} \\
&=( \lambda_h(u_h), u - \psi )_{-1,1} + ( \lambda_h(u_h), \psi - \tilde{u_h} )_{-1,1} - ( \lambda(u) - \lambda_h(u_h), \tilde{u_h} - (u_h^0)^c )_{-1,1} \\ 
&\leq ( \lambda_h(u_h), \psi - \tilde{u_h} )_{-1,1} - ( \lambda(u) - \lambda_h(u_h), \tilde{u_h} - (u_h^0)^c )_{-1,1}.
\end{align*}
Since $\psi - \tilde{u_h}= - ((u_h^0)^c - \psi)^{+}$ and recalling that $\lambda_h(u_h) = 0 $ on triangles $T \in \mathcal{N}_h,$ we deduce 
\begin{align*}
( \lambda_h(u_h)-\lambda(u), u &- (u_h^0)^c )_{-1,1} &\leq -\sum_{T \in \mathcal{C}_h}  \int_T \lambda_h(u_h) ((u_h^0)^c - \psi)^{+} ~dx - ( \lambda(u) - \lambda_h(u_h), \tilde{u_h} - (u_h^0)^c )_{-1,1}.
\end{align*}
Finally, using Young's inequality, we obtain
\begin{align*}
( \lambda_h(u_h)-\lambda(u), u - (u_h^0)^c )_{-1,1} &\leq -\sum_{T \in \mathcal{C}_h}  \int_T \lambda_h(u_h) ((u_h^0)^c - \psi)^{+} ~dx + \epsilon \| \lambda(u) -\lambda_h(u_h)\|_{H^{-1}(\Omega)}^2 \\&+ \frac{1}{\epsilon}  \|\nabla (\psi-(u_h^0)^c)^+\|^2_{L^2(\Omega)}
\end{align*}
for some $\epsilon >0,$ thereby, proving the claim. 
\end{proof}

\subsection{Local Efficiency Estimates}

In this subsection, we present the local efficiency estimates, which can be derived using the standard bubble function technique. To begin, we define \( \text{Osc}(f, \Omega) \) as:  
\begin{align*}
\text{Osc}(f, \Omega) := \sum_{T \in \mathcal{T}_h} \text{Osc}(f, T),
\end{align*}
where  
\begin{align*}
\text{Osc}(f, T) := \bigg(h_T^2 \|f - \overline{f}\|^2_{L^2(T)}\bigg)^{\frac{1}{2}}.
\end{align*}
Here, \( \bar{f} \) denotes the \( L^2 \)-projection of $f$ onto the space of piecewise constant functions on $\Omega$.
\par
\noindent
\textit{(i) Bound for $\eta_1$:} 
Let \( T \in \mathcal{T}_h \) be an arbitrary element, and let \( \kappa \in \mathbb{P}_3(T) \) denote the bubble function associated with \( T \). This function vanishes on the boundary \( \partial T \) and takes unit value at the barycenter of \( T \). Using the equivalence of norms in finite dimensional spaces, we can establish the following estimate:
\begin{align}
\|{\bar{f}} + \nabla \cdot(\nabla_wu_h)-\lambda_h(u_h)\|^2_{{L^2}(T)} 
&\lesssim \int_T \kappa \big({\bar{f}} + \nabla \cdot(\nabla_wu_h)-\lambda_h(u_h)\big)\big({\bar{f}}  + \nabla \cdot(\nabla_wu_h)-\lambda_h(u_h)\big) \, dx.
\end{align}

Next, define \( {\upsilon} := \kappa \big({\bar{f}} + \nabla \cdot(\nabla_wu_h)-\lambda_h(u_h)\big) \). Extend \( {\upsilon} \) by setting it to zero outside \( T \), so that \( {\upsilon} \in {H^1_0}(\Omega) \). Using integration by parts, relation \eqref{eq:Lambda} and a standard inverse estimate, we derive:
\begin{align}
\int_T \kappa \big({\bar{f}} &+ \nabla \cdot(\nabla_wu_h)-\lambda_h(u_h)\big)  \big({\bar{f}}  + \nabla \cdot(\nabla_wu_h)-\lambda_h(u_h)\big) \, dx \notag\\
&= \int\limits_T {f}  {\upsilon} \, dx + \int\limits_T ({\bar{f}} - {f})  {\upsilon} \, dx + \int\limits_T \nabla \cdot(\nabla_wu_h) {\upsilon} \, dx - \int\limits_T \lambda_h(u_h)  {\upsilon} \, dx  \notag \\
&= a({u}, {\upsilon}) + ( \lambda(u), \upsilon)_{-1,1,T} + \int\limits_T ({\bar{f}} - {f})  {\upsilon} \, dx - \int\limits_T \nabla_wu_h \cdot {\nabla \upsilon} \, dx - ( \lambda_h(u_h),  {\upsilon})_{-1,1,T}  \notag \\
&= \int\limits_T \big({\nabla}{u} - \nabla_wu_h) \cdot \nabla \upsilon \, dx + \int\limits_T ({\bar{f}} - {f})  {\upsilon} \, dx +(\lambda(u)-\lambda_h(u_h), {\upsilon})_{-1,1,T}.
\end{align}

Applying  Cauchy Schwarz inequality, we obtain:
\begin{align}
\int_T \kappa \big({\bar{f}} &+ \nabla \cdot(\nabla_wu_h)-\lambda_h(u_h)\big) \big({\bar{f}}  + \nabla \cdot(\nabla_wu_h)-\lambda_h(u_h)\big) \, dx \nonumber\\
&\lesssim \|\nabla u-\nabla_wu_h\|_{{L^2}(T)} |{\upsilon}|_{{H^1}(T)} + \|{\bar{f} - f}\|_{{L^2}(T)} \|{\upsilon}\|_{{L^2}(T)}+\|\lambda(u)-\lambda_h(u_h)\|_{H^{-1}(T)}|v|_{H^1(T)}.
\end{align}

Using the standard inverse estimate \( |{\upsilon}|_{{H^1}(T)} \lesssim h_T^{-1} \|{\upsilon}\|_{{L^2}(T)} \), we have:
\begin{align}
\|&{\bar{f}} + \nabla \cdot(\nabla_wu_h)-\lambda_h(u_h)\|^2_{{L^2}(T)} \notag \\
&\lesssim \bigg(h^{-1}_T \|\nabla u-\nabla_wu_h\|_{{L^2}(T)} + \|{\bar{f} - f}\|_{{L^2}(T)} +h^{-1}_T\|\lambda(u)-\lambda_h(u_h)\|_{H^{-1}(T)}\bigg)\|{\upsilon}\|_{{L^2}(T)} \notag \\
&\lesssim \bigg(h^{-1}_T \|\nabla u-\nabla_wu_h\|_{{L^2}(T)} + \|{\bar{f} - f}\|_{{L^2}(T)} +h^{-1}_T\|\lambda(u)-\lambda_h(u_h)\|_{H^{-1}(T)}\bigg)\|{\bar{f}} + \nabla \cdot(\nabla_wu_h)-\lambda_h(u_h)\|_{{L^2}(T)}. \notag
\end{align}
Finally, we establish:
\begin{align*}
h_T^2 \|{\bar{f}} + \nabla \cdot(\nabla_wu_h)-\lambda_h(u_h)\|^2_{{L^2}(T)} 
&\lesssim \|\nabla u-\nabla_wu_h\|^2_{{L^2}(T)} + h^2_T\|{\bar{f} - f}\|^2_{{L^2}(T)} +\|\lambda(u)-\lambda_h(u_h)\|^2_{H^{-1}(T)}.
\end{align*}
Using the triangle inequality, we conclude:
\begin{align*}
h_T^2 \|{{f}} + \nabla \cdot(\nabla_wu_h)-\lambda_h(u_h)\|^2_{{L^2}(T)} 
&\lesssim \|\nabla u-\nabla_wu_h\|^2_{{L^2}(T)} + (Osc(f,T))^2 +\|\lambda(u)-\lambda_h(u_h)\|^2_{H^{-1}(T)}.
\end{align*}

\par
\noindent
\textit{(ii) Bound for $\eta_2$:} 
Consider an interior edge \( e \in \mathcal{E}^i_h \) shared by two neighbouring elements \( T^{-} \) and \( T^{+} \), such that \( e = \partial T^{-} \cap \partial T^{+} \). Let \( \b{n_e} \) represent the outward unit normal vector to \( e \), oriented from \( T^{-} \) to \( T^{+} \). We define a bubble function \( \kappa \in \mathbb{P}_2(T^{-} \cup T^{+}) \), which vanishes on the boundary of the quadrilateral formed by \( T^{-} \cup T^{+} \) and equals 1 at the midpoint of \( e \). Define ${\beta} := \kappa {\kappa_{1}} \quad \text{on } T^{-} \cup T^{+},$
where \( {\kappa_{1}} \in [\mathbb{P}_0(T^{-} \cup T^{+})] \) is given by \( {{\kappa_{1}}} = \sjump{\nabla_wu_h} \) on \( e \). Extend \( {\beta} \) by zero outside \( T^{-} \cup T^{+} \), ensuring \( {\beta} \in [{H^1_0}(\Omega)] \). Using the equivalence of norms in finite dimensional spaces, we obtain
\begin{align}\label{pp5:eqn1}
    \|{\kappa_{1}}\|^2_{{L^2}(e)} &\lesssim \int\limits_e \kappa {\kappa_{1}} \cdot {\kappa_{1}}~ds 
    = \int\limits_e {\beta} \cdot {\kappa_{1}}~ds.
\end{align}
Applying integration by parts, relation \eqref{eq:Lambda}, Cauchy Schwarz inequality, and standard inverse estimates, we obtain
\begin{align}\label{eqn1}
    \int_e {\sjump{\nabla_wu_h} } \cdot {\beta}~ds &= \int\limits_{T^{-} \cup T^{+}} {\nabla_wu_h}\cdot \nabla \beta~dx + \int\limits_{T^{-} \cup T^{+}} \nabla\cdot(\nabla_wu_h)\beta~dx \nonumber \\
    &= \int\limits_{T^{-} \cup T^{+}} {\nabla_wu_h}\cdot \nabla \beta~dx + \int\limits_{T^{-} \cup T^{+}} (\nabla\cdot(\nabla_wu_h) + f - \lambda_h(u_h))\beta~dx \nonumber  \\
    &\quad - \int\limits_{T^{-} \cup T^{+}} {\nabla u_h}\cdot \nabla \beta~dx -(\lambda(u),\beta)_{-1,1,T^{+}\cup T^{-}}+(\lambda_h(u_h),\beta)_{-1,1,T^{+}\cup T^{-}} \nonumber \\
    &\leq \sum_{T \in \mathcal{T}_e} \bigg(  \|\nabla u -\nabla_w u_h\|_{{L^2}(T)} |{\beta}|_{{H^1}(T)} + \|\nabla\cdot(\nabla_wu_h) + f - \lambda_h(u_h)\|_{{L^2}(T)} \|{\beta}\|_{{L^2}(T)}  \nonumber \\
    &+ \|\lambda(u)-\lambda_h(u_h)\|_{H^{-1}(T)}|{\beta}|_{{H^1}(T)}\bigg) \nonumber \\
      &\leq \sum_{T \in \mathcal{T}_e} \bigg(  \|\nabla u -\nabla_w u_h\|_{{L^2}(T)} h^{-1}_T + \|\nabla\cdot(\nabla_wu_h) + f - \lambda_h(u_h)\|_{{L^2}(T)}   \nonumber \\
    &+ h^{-1}_T\|\lambda(u)-\lambda_h(u_h)\|_{H^{-1}(T)}\bigg)\|{\beta}\|_{{L^2}(T)} \nonumber \\  
   &\leq \sum_{T \in \mathcal{T}_e} \bigg(  \|\nabla u -\nabla_w u_h\|_{{L^2}(T)} h^{-1}_T + \|\nabla\cdot(\nabla_wu_h) + f - \lambda_h(u_h)\|_{{L^2}(T)} \nonumber   \\
    &+ h^{-1}_T\|\lambda(u)-\lambda_h(u_h)\|_{H^{-1}(T)}\bigg)h_e^{\frac{1}{2}}\|{\kappa_{1}}\|_{{L^2}(e)}.
\end{align}
Combining the bounds obtained in equation \eqref{pp5:eqn1} and \eqref{eqn1}, we obtain
\begin{align*}
    h_e^{\frac{1}{2}} \|\sjump{\nabla_wu_h}\|_{{L^2}(e)} &\lesssim   \|\nabla u -\nabla_w u_h\|_{{L^2}(\mathcal{T}_e)} + h_e\|\nabla\cdot(\nabla_wu_h) + f - \lambda_h(u_h)\|_{{L^2}(\mathcal{T}_e)} +\|\lambda(u)-\lambda_h(u_h)\|_{H^{-1}(\mathcal{T}_e)}.
\end{align*}
Thus, the estimate \((ii)\) follows from \((i)\).
\\
\par
\noindent
\textit{(iii) Bound for $-\int_T \lambda_h(u_h) ((u_h^0)^c - \psi)^{+} ~dx,~T \in \mathcal{C}_h$:} 
\begin{align*}
-\int\limits_T \lambda_h(u_h) &((u_h^0)^c - \psi)^{+} ~dx \nonumber \\& \leq  \int\limits_T (f+\nabla\cdot\nabla_wu_h-\lambda_h(u_h))((u_h^0)^c - \psi)^{+}~dx  +\int\limits_T (-f)((u_h^0)^c - \psi)^{+}~dx \\
& \leq \| f+\nabla\cdot\nabla_wu_h-\lambda_h(u_h) \|_{L^2(T)}   \|((u_h^0)^c - \psi)^{+}\|_{L^2(T)}+\int\limits_T (-f)((u_h^0)^c - \psi)^{+}~dx \\
& \leq \|\nabla u-\nabla_wu_h\|_{{L^2}(T)} + Osc(f,T) +\|\lambda(u)-\lambda_h(u_h)\|_{H^{-1}(T)} \\&+ h_T^{-1} \|(u_h^0)^c - \psi)^{+}\|_{L^2(T)} + \int\limits_T (-f )((u_h^0)^c - \psi)^{+}~dx.
\end{align*}
\par
\noindent
\textit{(iv) Bound for $\|\nabla (\psi-(u_h^0)^c)^+\|^2_{L^2(T)}$:} 
\begin{align*}
\|\nabla (\psi-\tilde{u}_h)^+\|^2_{L^2(T)} 
&\lesssim \int\limits_{T}(\nabla{\psi}-\nabla{\bar{\psi}})^2~dx + \int\limits_{T}(\nabla({\bar{\psi}} - u_h)^+)^2~dx + \int\limits_{T}(\nabla u_h-\nabla \tilde{u}_h)^2~dx  \\
&\lesssim \|\nabla{\psi}-\nabla{\bar{\psi}}\|^2_{L^2(T)} + \eta_3^2 + \int\limits_{T}(\nabla({\bar{\psi}} - u_h)^+)^2~dx,
\end{align*}
where, the last inequality follows from Lemma \ref{decom}.
\section{Numerical Experiments}
In this section, we conduct a series of numerical experiments to evaluate the effectiveness of the error indicator \eqref{eq:Est1} and the adaptive modified weak Galerkin method. All simulations are implemented using MATLAB (version R2020b). The adaptive algorithm follows the paradigm outlined below:
\begin{center}
\textbf{SOLVE} $\longrightarrow$ \textbf{ESTIMATE} $\longrightarrow$ \textbf{MARK} $\longrightarrow$ \textbf{REFINE}
\end{center}
The \textbf{SOLVE} step involves calculating the discrete solution ${u_h}$ by solving the discrete variational inequality \eqref{discrete} using the primal-dual active set strategy \cite{Primal_obstacle}. Following this, the \textbf{ESTIMATE} step computes the elementwise error estimator $\eta_h$ as described in Section 4. In the \textbf{MARK} step, we apply D\"orfler's marking strategy \cite{Dorfler:1996:Afem} with a parameter $\theta = 0.4$ to identify the elements of the triangulation that need refinement. Finally, in the \textbf{REFINE} step, the marked elements are refined using the newest vertex bisection algorithm \cite{Verfurth:1994:Adaptive}, generating a new mesh. This process is iterated until the desired level of accuracy is achieved.
\par 
\noindent
We start by constructing the conforming component of the discontinuous solution, denoted as $(u_h^0)^c \in H^1(\Omega)$. While there are multiple methods to approximate this component, we employ an averaging technique. Specifically, it is defined as follows:
\begin{align*}
(u_h^0)^c(p) & := \dfrac{1}{|\mathcal{T}_p|} \sum_{T \in \mathcal{T}_p} u_h^0|_T(p), ~u_h \in V_h^{0w}.
\end{align*}
The first model problem, inspired by \cite{Nochetto:2005}, is designed such that the exact solution $u$ is not known \textit{a priori}. In contrast, the exact solutions for the subsequent two examples are explicitly known. In all examples, we numerically confirm that the estimators converge at the optimal rate of $\frac{1}{\sqrt{\text{Ndof}}}$, where Ndof represents the number of degrees of freedom.
\begin{example} Consider the domain $\Omega = (-2,2)\times (-1,1)$ with homogeneous boundary data. We introduce two constant load functions $f=0$ and $f=-15$, and a smooth obstacle defined by
\begin{align*}
\psi =10-6(x^2-1)^2-20 (r^2-x^2), \quad r^2=x^2+y^2.
\end{align*}
\end{example}
Figure \ref{error_estimator_eg2.jpg} demonstrates that the error estimator $\eta_h$ converges optimally for both load cases, $f=0$ and $f=-15$ with the increase in degrees of freedom. The corresponding adaptive meshes at refinement level 21 are illustrated in Figure \ref{Adpative mesh1}. The obstacle’s profile can be interpreted as two hills joined by a saddle. The contact region is observed to evolve with increasing load function $f$.
\par
\noindent
\begin{figure}
	\begin{subfigure}[b]{0.45\textwidth}
		\includegraphics[width=\linewidth]
		{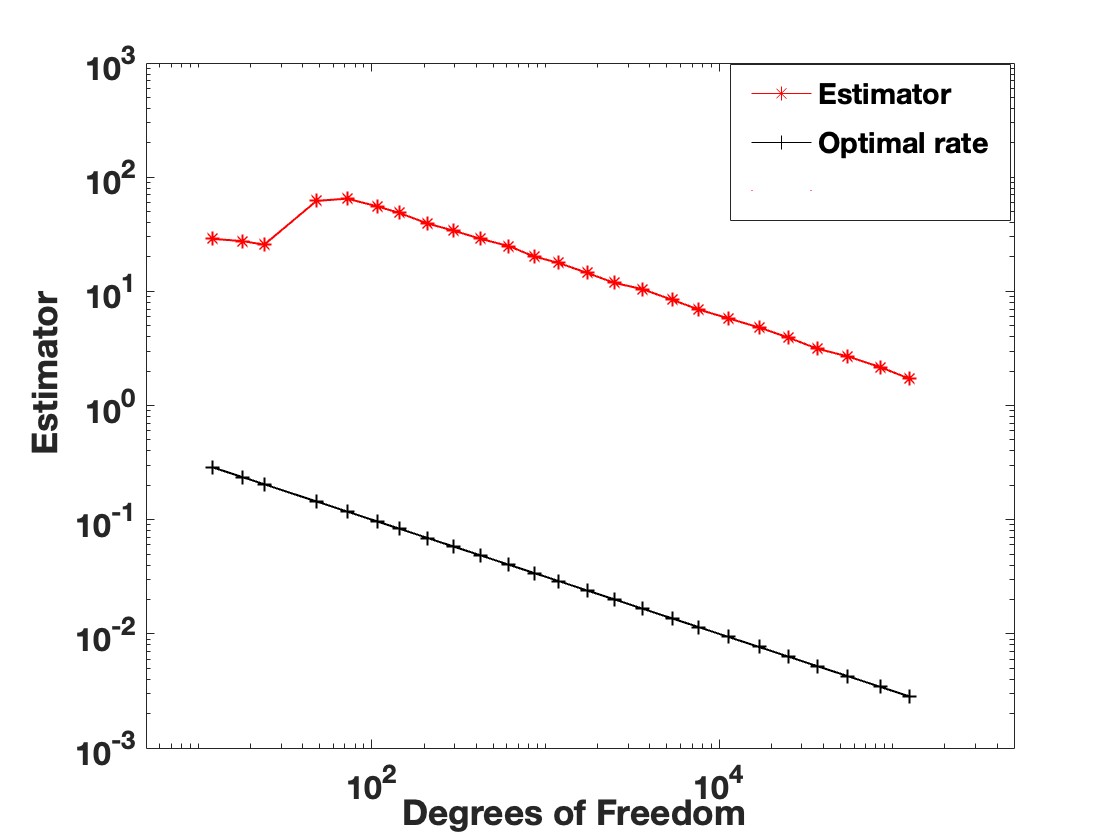}
		\caption{$f=0$}
	\end{subfigure}
	\begin{subfigure}[b]{0.45\textwidth}
		\includegraphics[width=\linewidth]{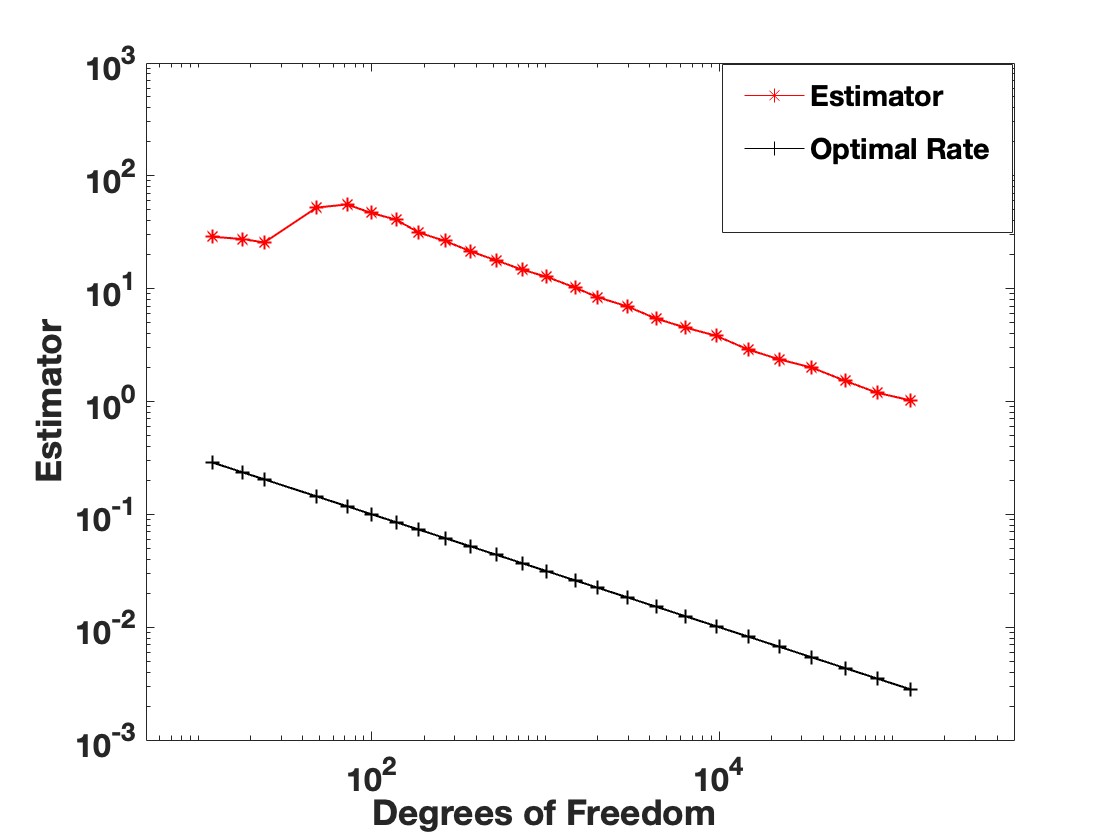}
		\caption{$f=-15$}
	\end{subfigure}
	\caption{Plot of error and residual error estimator  for Example 5.1.}\label{error_estimator_eg2.jpg}
\end{figure}
\begin{figure}
	\begin{subfigure}[b]{0.45\textwidth}
		\includegraphics[width=\linewidth]
		{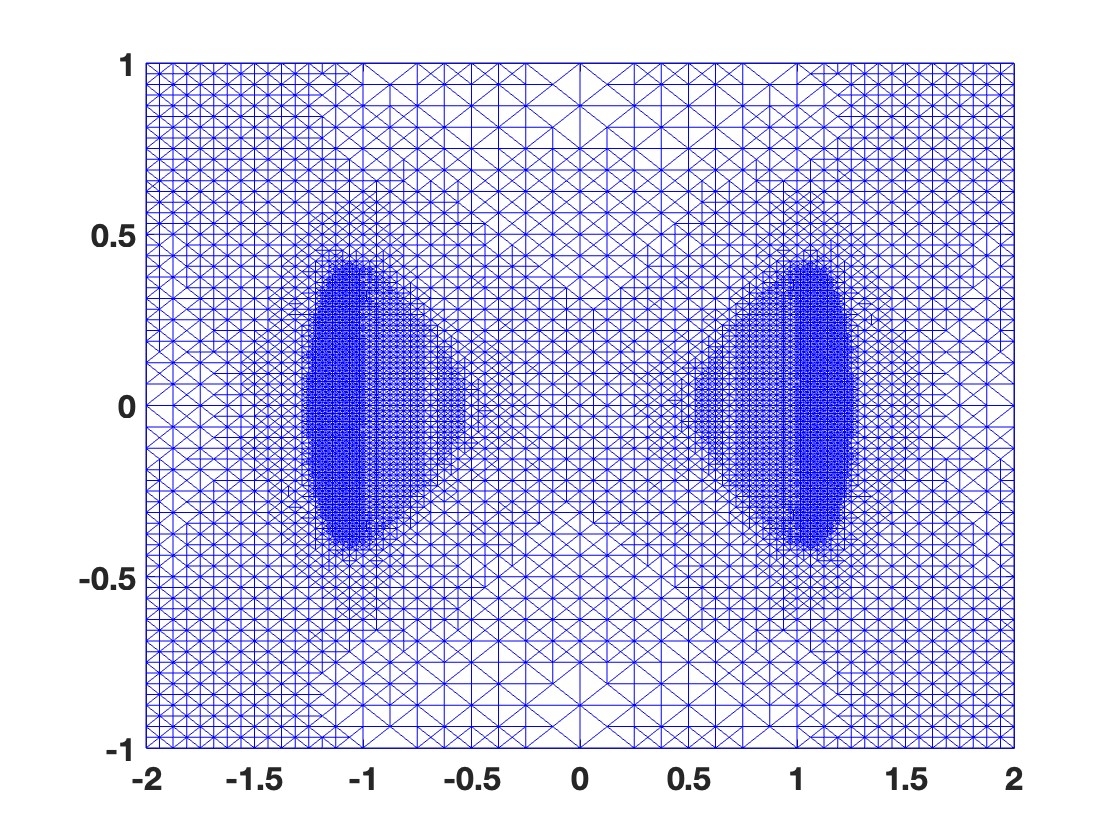}
		\caption{$f=0$}
	\end{subfigure}
	\begin{subfigure}[b]{0.45\textwidth}
		\includegraphics[width=\linewidth]{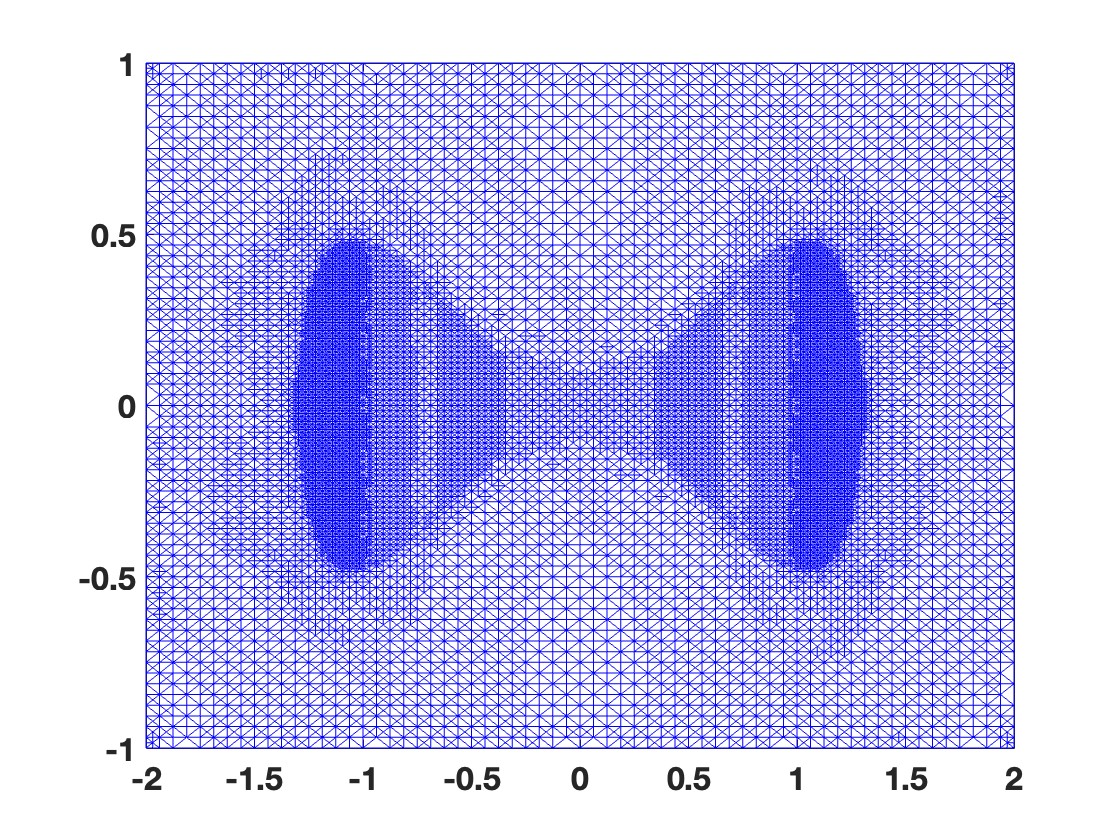}
		\caption{$f=-15$}
	\end{subfigure}
	\caption{Plot of adaptive mesh at certain refinement for Example 5.1.}\label{Adpative mesh1}
\end{figure}

\begin{figure}
		\includegraphics[height=9cm,width=10cm]
		{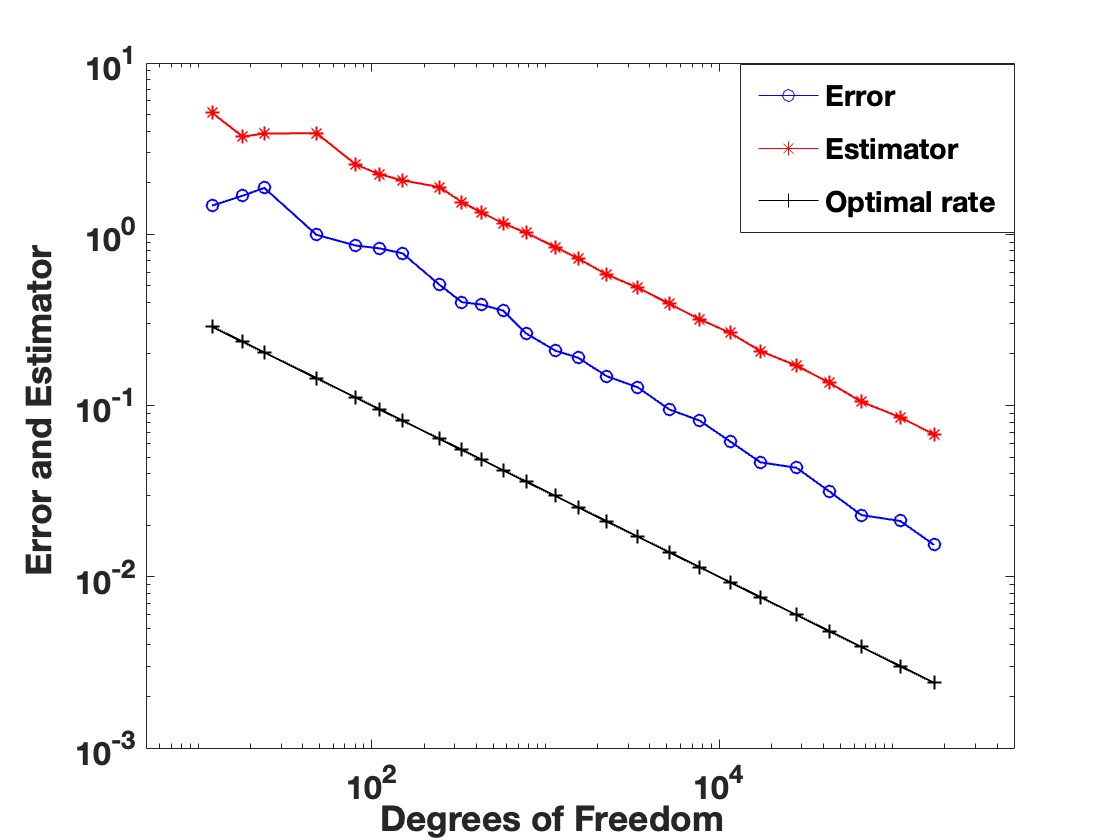}
	\caption{Plot of error and residual error estimator  for Example 5.2.}\label{error_estimator_eg1}
\end{figure}

\begin{figure}
		\includegraphics[height=9cm,width=10cm]
		{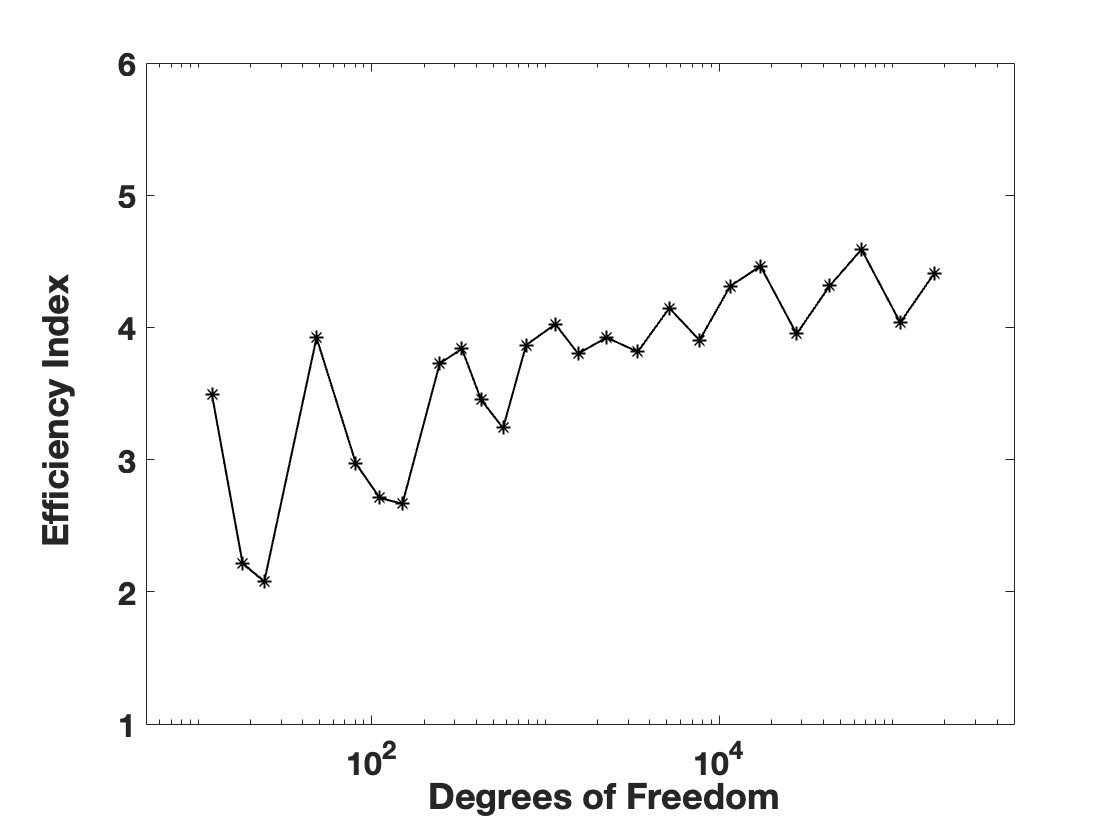}
	\caption{Plot of efficiency index for Example 5.2.}\label{efficiency_eg1}
\end{figure}

\begin{figure}
		\includegraphics[height=9cm,width=10cm]
		{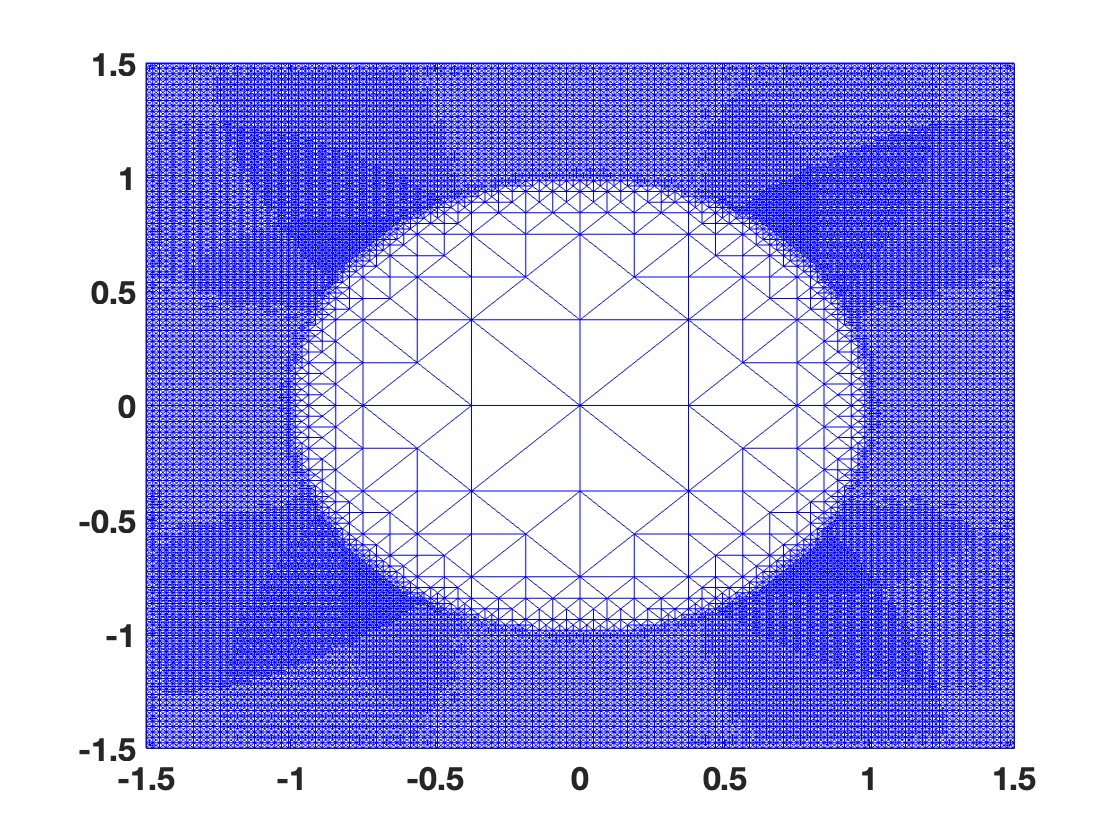}
	\caption{Plot of adaptive mesh at certain refinement for Example 5.2.}\label{adaptive_mesh_eg1}
\end{figure}

\begin{figure}
		\includegraphics[height=9cm,width=10cm]
		{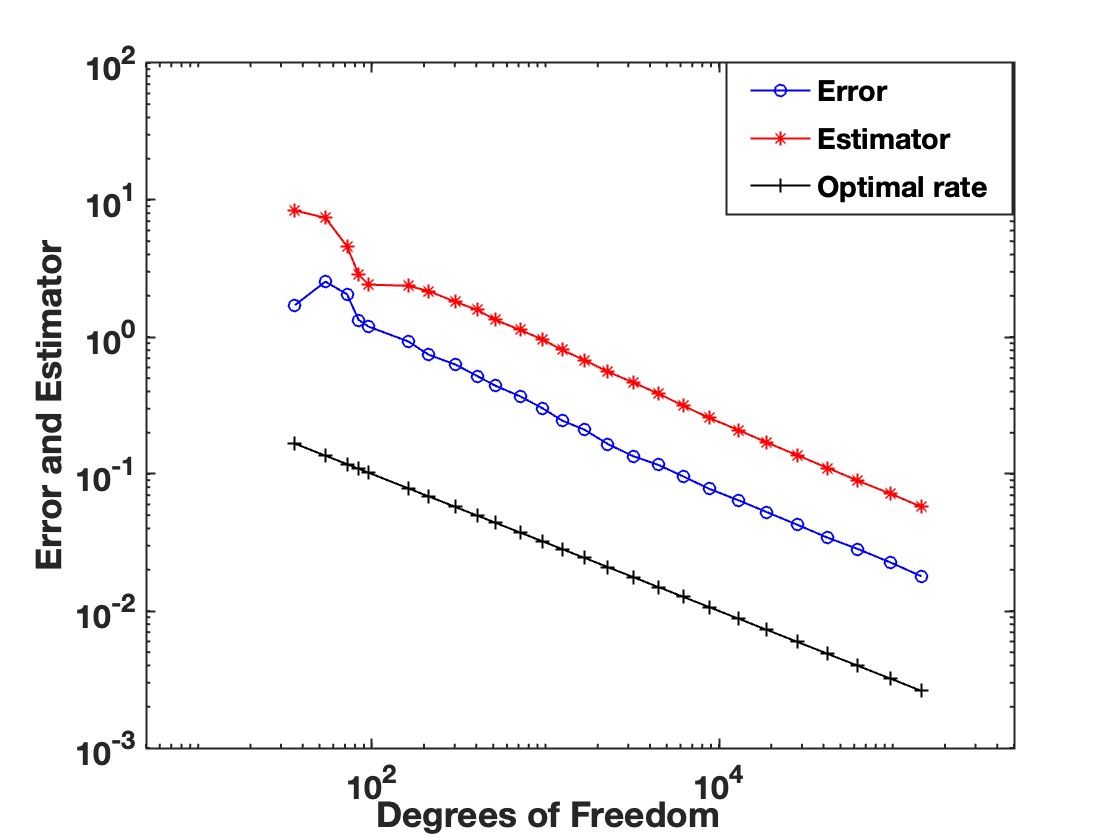}
	\caption{Plot of error and residual error estimator  for Example 5.3.}\label{error_estimator_eg4}
\end{figure}

\begin{figure}
		\includegraphics[height=9cm,width=10cm]
		{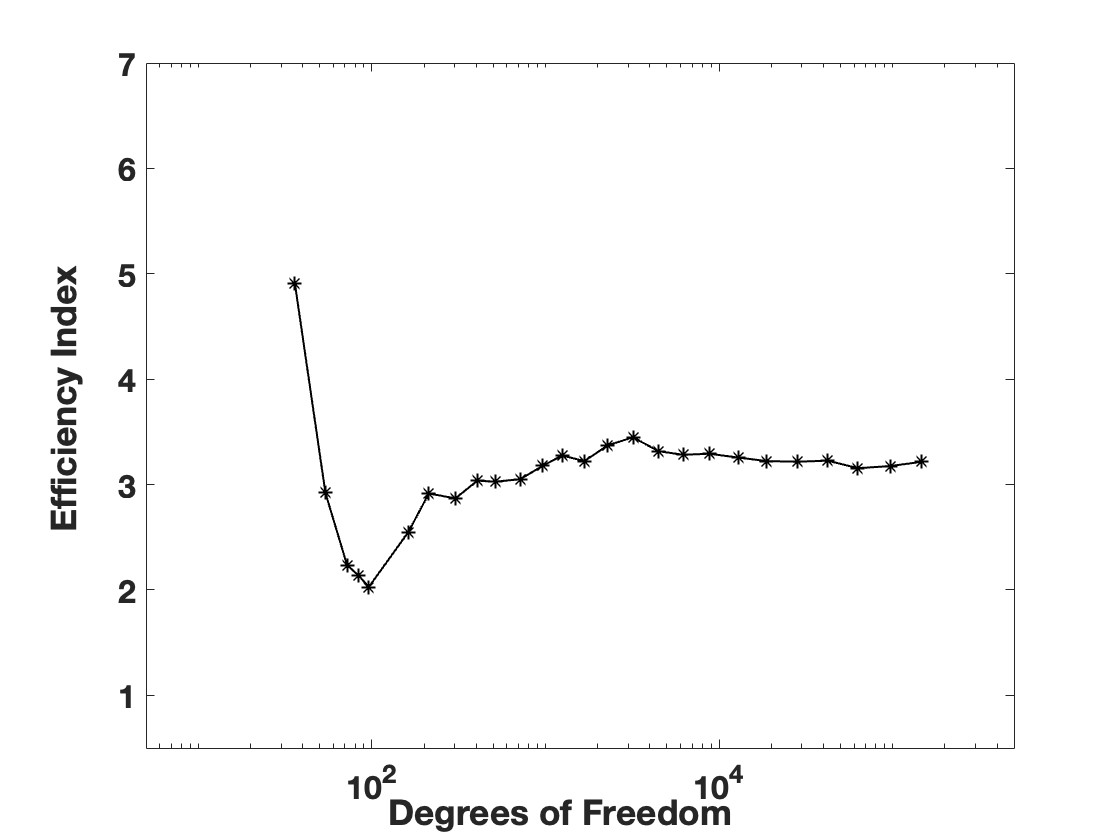}
	\caption{Plot of efficiency index for Example 5.3.}\label{efficiency_eg4}
\end{figure}

\begin{figure}
		\includegraphics[height=9cm,width=10cm]
		{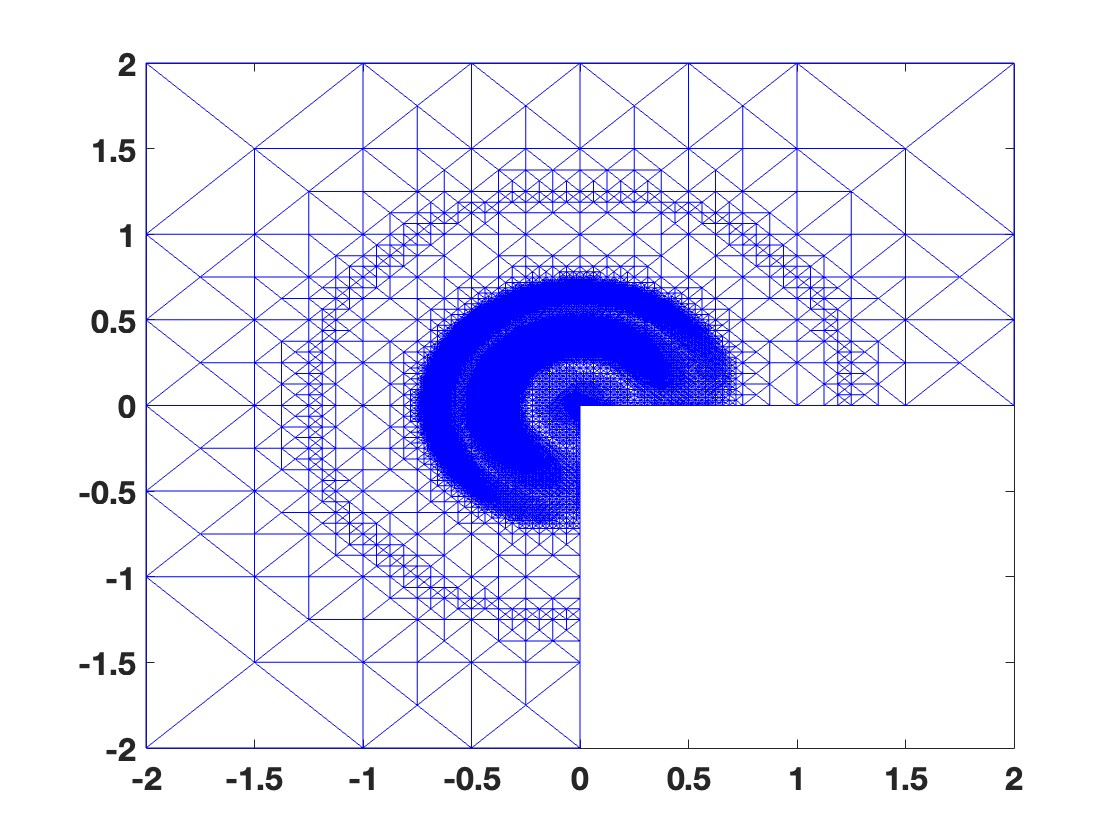}
	\caption{Plot of adaptive mesh at certain refinement for Example 5.3.}\label{adaptive _mesh_eg4}
\end{figure}

\begin{example} \label{ch5:ex}
	Let $\Omega=(-1.5,-1.5)^2$, $f=-2$, and
	$\chi:=0$. The exact solution is known and given by (taken from \cite{Bartels:2004}):
	\begin{equation*}
	u := \left\{ \begin{array}{ll} r^2/2-\text{ln}(r)-1/2, & r\geq 1\\\\
	0, & \text{otherwise}.
	\end{array}\right.
	\end{equation*}
	where
	$r^2=x^2+y^2$ for $(x,y)\in \mathbb{R}^2$.\\
\end{example}

Figure \ref{error_estimator_eg1} illustrates the optimal order decay of both the error and the estimator as the number of degrees of freedom increases, for Example \ref{ch5:ex}. This further corroborates the reliability of the estimator $\eta_h$. The efficiency index, calculated as the ratio of the estimator to the error, is shown in Figure \ref{efficiency_eg1}. The graph of the efficiency index is bounded both above and below by a generic constant, which guarantees that the residual estimator $\eta_h$ is not only reliable but also efficient. Finally, Figure \ref{adaptive_mesh_eg1} presents an adaptive mesh at a particular refinement level. As expected, we observe a higher level of refinement in the free boundary region.

\begin{example} \label{ch5:ex2} In the next model example motivated form \cite{Braess:2007}, we consider the following data over a L-shaped domain 
{	\small
	\begin{align*}
	\Omega &= (-2,2)^2 \setminus [ 0, 2) \times (-2, 0], \quad \chi=0, \\
	u &= r^{2/3} sin(2\theta/3 ) \gamma_1(r),\\
	f&=-r^{2/3} sin(2 \theta/3) \Big(\frac{\gamma_1^{'}(r)}{r} + \gamma_1^{''}(r) \Big)-\frac{4}{3}r^{-1/3}sin(2 \theta /3) \gamma_1^{'}(r)-\gamma_2(r),\\
	\text{where}\\
	\gamma_1(r) &= \begin{cases}  1, \quad  \tilde{r}<0 \\
	-6 \tilde{r}^5 + 15 \tilde{r}^4 -10 \tilde{r}^3 +1 , \quad  0 \leq \tilde{r} <1 \\
	0  , \quad \tilde{r} \geq 1,\\
	\end{cases}\\
	\gamma_2(r) &= \begin{cases} 0, \quad r \leq \frac{5}{4} \\
	1, \quad \mbox{otherwise}
	\end{cases}\\
	\text{with}\quad
	\tilde{r}= 2(r-1/4).
	\end{align*}}
\end{example}
Figure \ref{error_estimator_eg4} shows the behaviour of both the true error and the error estimator. It is observed that both the error and the estimator converge at the optimal rate with respect to the degrees of freedom. The efficiency of the error estimator is demonstrated in Figure \ref{efficiency_eg4}. In Figure \ref{adaptive _mesh_eg4}, the adaptively refined grid, steered by the residual error estimator for the MWG methods, is presented. As anticipated, the mesh is refined most heavily around the regions where the solution exhibits singular behaviour.

\par 
\noindent
 
\textbf{Acknowledgements:} The author gratefully acknowledges Prof. Kamana Porwal and Dr. Ritesh Singla for their valuable and insightful suggestions.\\
\par 
\noindent
\textbf{Declarations:} This manuscript has no associated data.

\end{document}